\documentclass{amsart}
\usepackage{amssymb,latexsym}
\usepackage{footnote}
\usepackage{color}
\usepackage{longtable}
\usepackage{booktabs}
\usepackage{hyperref}

\usepackage{todonotes}
\usepackage{soul}
\setstcolor{red}

\newcommand{\pmmtodo}[1]{\todo[color=orange!30,bordercolor=orange,size=\small]{PMM: {#1}}}

\theoremstyle{plain}
\newtheorem{theorem}{Theorem}
\newtheorem{corollary}[theorem]{Corollary}
\newtheorem{proposition}[theorem]{Proposition}

\newtheorem{lemma}[theorem]{Lemma}
\newtheorem{question}[theorem]{Question}

\theoremstyle{definition}
\newtheorem{definition}[theorem]{Definition}
\newtheorem{example}[theorem]{Example}
\newtheorem{notation}[theorem]{Notation}
\newtheorem{remark}[theorem]{Remark}

    \def\ptk{[l]}
    \def\Zfl{Z_{F,l}}
    \def\SV{\bar{S}}
    \def\SW{\SV_{\mathit{l}}}
    \def\SVdual{\bar{T}}
    \def\SWdual{\SVdual_{\mathit{l}}}

    \def\Spec{\operatorname{Spec}}
    \def\Cactus{\operatorname{Cactus}}

\newcommand{\degF}{\ensuremath{d}}
\newcommand{\df}{\ensuremath{\mathrm{Diff}}}
\newcommand{\divp}{\ensuremath{\mathsf{dp}}}
\newcommand{\h}[1]{\-\mbox{-#1}}
\newcommand{\KK}{K}
\newcommand{\PP}{\ensuremath{\mathbb{P}}}

\newcommand{\ZGl}{Z_{G, x_0}}

\DeclareMathOperator{\linearOp}{Lin}
\DeclareMathOperator{\ord}{ord}
\DeclareMathOperator{\ls}{\textsc{ls}}
\newcommand{\linear}[2]{\linearOp(#1)^{#2}}%
\newcommand{\Ver}[1]{V_{#1}}%

\title[]{On polynomials with given Hilbert function and applications} 

\author[A. Bernardi, J. Jelisiejew, P. Macias Marques, K. Ranestad]{Alessandra Bernardi, Joachim Jelisiejew, Pedro Macias Marques, Kristian Ranestad}

\address{Dipartimento di Matematica, Universit\`{a} di Trento, via Sommarive 14, I-38123 Povo (Trento), Italy.}
\email{alessandra.bernardi@unitn.it}

\address{Faculty of Mathematics, Informatics, and Mechanics, University of
Warsaw, Banacha 2, 02-097 Warszawa, Poland}
\email{jjelisiejew@mimuw.edu.pl}

\address{Departamento de Matem\'{a}tica, Escola de Ci\^{e}ncias e Tecnologia, Centro de Investiga\c{c}\~{a}o em Matem\'{a}tica e Aplica\c{c}\~{o}es, Instituto de Investiga\c{c}\~{a}o e Forma\c{c}\~{a}o Avan\c{c}ada, Universidade de \'{E}vora, Rua Rom\~{a}o Ramalho, 59, P--7000--671 \'{E}vora, Portugal}
\email{pmm@uevora.pt}

\address{Matematisk institutt, Universitetet i Oslo, PO Box 1053, Blindern, NO-0316 Oslo, Norway}
\email{ranestad@math.uio.no}



\begin{document}

\begin{abstract}
    Using Macaulay's correspondence we study the family of Artinian Gorenstein local
    algebras with fixed symmetric Hilbert function decomposition.
    As an application we give a new lower bound for cactus varieties of the third
    Veronese embedding. We discuss the case of cubic surfaces, where
interesting phenomena occur.
\end{abstract}

\maketitle

\section*{Introduction}

\bigskip
 Macaulay established a correspondence between polynomials and Artinian local Gorenstein algebras.  In particular, any polynomial is a dual socle generator of an Artinian local Gorenstein algebra.  In this paper we interpret the Hilbert function of the algebra as a Hilbert function of the corresponding polynomial, and give a description of the set of polynomials with given symmetric Hilbert function decomposition, in a fixed polynomial ring. We consider polynomials $f$ in a divided power ring ${S=\KK_{\divp}[x_1,\ldots,x_n]}$, and a polynomial ring ${T=\KK[y_1,\ldots ,y_n]}$ acting on $S$ by contraction (see Section~\ref{nota}).  The Artinian local Gorenstein algebra $A$ associated to $f\in S$ is the quotient  $T/f^{\bot}$ where $f^{\bot}$ is the annihilator ideal of $f$. Thus $\Spec(T/f^{\bot})\subset \Spec(T)$ is a local Gorenstein scheme supported at the origin of the affine space $\Spec(T)$.

The application we have in mind is that of apolarity and the dimension of cactus varieties of cubic forms.  Cactus varieties are generalizations of secant varieties.
  
\begin{definition}
Let $X\subset \PP^N$  be a projective variety.  The $r$-th cactus variety $\Cactus_r(X)$ of $X$ is the closure of the union of the linear spaces spanned by length $r$ subschemes on $X$.
\end{definition}
Here we abuse  slightly the notation of variety, since the cactus variety is
often a reducible algebraic set. We are interested in the case when $X
    \simeq \mathbb{P}^n$ is embedded into $\mathbb{P}^{\binom{n+3}{3}-1}$ by
the third Veronese embedding.

Consider, like above, a divided power ring ${\SV=\KK_{\divp}[x_0,\ldots ,x_n]}$, and a polynomial ring ${\SVdual=\KK[y_0,\ldots ,y_n]}$ acting on $\SV$ by contraction.
A cubic form $F\in \SV_3$ up to multiplication
by scalars is a point in $\PP(\SV_3)$.  The pure cubes form a $n$-dimensional
subvariety ${\Ver{3,n}\subset \PP(\SV_3)}$. The least length of a subscheme $Z\subset \Ver{3,n}$ whose linear span contains $F$ is called the \emph{cactus rank} of $F$.  The closure of the set of cubic forms with cactus rank $r$ is the $r$-th cactus variety of  $\Ver{3,n}$, denoted $\Cactus_r(\Ver{3,n})$. Via the contraction action, $\SVdual$ is the natural homogeneous coordinate ring on $\PP(\SV_1)$, and a $Z\subset \Ver{3,n}$ contains $F$ in its span, if and only if its homogeneous ideal $I_Z\subset \SVdual$ is contained in $F^{\bot}$. This classical fact is called the \emph{apolarity lemma} and is the motivation for defining a subscheme $Z\subset \PP(\SV_1)$ \emph{apolar} to $F$ if $I_Z\subset F^{\bot}$.

We apply Macaulay's correspondence to investigate local
Gorenstein schemes  that are apolar to $F$.  Our main result is the
following lower bound
on the dimension of cactus varieties of cubic forms.
\begin{theorem}[Corollary~\ref{TheCorollary}]
    Let $n\geq 8$ and $18\leq r\leq 2n+2$ and let $V_{3,n}$ be the third Veronese embedding of $\PP^n$ in $\PP^N$, with ${N=\tbinom{n+3}{3}}$. Then
    \[\dim  \Cactus_{r}(\Ver{3,n})\geq
\begin{cases}
(rn+r-1)+\frac{r(r-2)(r-16)}{48}-1 &\text{if } r\geq 18\; \text{even},\\
(rn+r-1)+\frac{(r-1)(r-3)(r-17)}{48}-2&\text{if } r\geq 19\; \text{odd}.\\
\end{cases}
\]
Hence, under these assumptions the $r$-th secant variety of $V_{3,n}$ is strictly contained in the $r$-th cactus variety of $V_{3,n}$.
\end{theorem}
The $r$-th secant variety of $V_{3,n}$ is ${\sigma_r(\Ver{3,n})=\overline{\bigcup_{P_1,\ldots,P_r\in V_{3,n}}\langle P_1,\ldots,P_r \rangle}}$. The fact that the inclusion of $\sigma_r(\Ver{3,n})$ in  $\Cactus_{r}(\Ver{3,n})$ is strict is a consequence of the inequality
\[\dim\sigma_r(\Ver{3,n})\le r n+r-1.\]
The righ\h{hand} side is the expected dimension of the secant variety, and it is an easy parameter count that gives an upper bound. The actual dimension of of the secant variety is known thanks to the Alexander and Hirschowitz Theorem \cite{AH}.

For $r \geq2n+2$ the variety $\Cactus_{r}(\Ver{3, n})$
fills the ambient space, see \cite{BR}. Observe that ${\sigma_r(\Ver{3,n})=\Cactus_{r}(\Ver{3,n})}$, for ${r\le13}$, see \cite{CN11} for the cases ${r\le11}$, and \cite{CJN} for the remaining cases.

\medskip
 
The link between the Artinian local Gorenstein algebras and apolar schemes to a homogeneous form $F\in\SV$ is provided by the fact that if ${f=F(1,x_1,\ldots,x_n)}$, then 
\[\Spec(K[y_1,\ldots,y_n]/f^{\bot})\subset \PP(\SV_1)\] 
is a local scheme supported at $[1:0:\cdots:0]$ that is apolar to $F$. We define 

\begin{definition}\label{LocalCactusRank}  The minimal length of a local apolar scheme to $F$ is called the \emph{local cactus rank} of $F$.
\end{definition}
That link is strengthened by the following result, where we call the sum of
the homogeneous terms of a polynomial $g$ of degree at most $\degF$ the
\emph{degree-$\degF$ tail} of $g$. 
\begin{proposition}\label{tail}
    Let $F$ be a homogeneous polynomial of degree $\degF$, and let
    ${f=F(1,x_1,\ldots,x_n)}$.
    Let $\Gamma$ be a zero\h{dimensional} scheme of minimal length among local schemes supported at
    ${[l] = [1:0:\cdots : 0]}$ that are apolar to $F$.  Then $\Gamma$ is the affine apolar
    scheme of a polynomial $g$ whose degree-$\degF$ tail equals $f$.
\end{proposition}
A particularly important problem is to find  the cactus rank of a general form in $\SV_\degF$, i.e.\ the minimal $r$ such that ${\Cactus_r(\Ver{\degF,n})=\PP(\SV_\degF)}$.
Our results do not improve previous known bounds, so this remains a major open problem in this theory. We refer the interested reader to \cite{IK}, \cite{BR}, \cite{BB}, and \cite{BBM}.

\medskip

Now, the first step in order to be able to compute the cactus rank, is to describe the structure of minimal apolar schemes. If we start by considering a minimal zero\h{dimensional} scheme $\Gamma$ apolar to a form $F$ and decompose it as ${\Gamma=\Gamma_1\cup\cdots\cup\Gamma_k}$, where each $\Gamma_i$ is a scheme supported on one point, then we take the corresponding decomposition ${F=F_1+ \cdots + F_k}$, where $\Gamma_i$ is a minimal local scheme apolar to $F_i$ for ${i=1, \ldots , k}$. According to Proposition~\ref{tail}, one would like to find invariants for
local apolar Gorenstein schemes, parameterizing the degree-$d$ tails of all
polynomials that define a scheme with given invariant. Iarrobino's analysis
(\cite{I}), that we recall in the first section, provides one such discrete
invariant: the symmetric Hilbert function decomposition. Then, if one wants to
estimate the dimension of the sets of polynomials with the same local cactus
rank, one needs to understand the structure of the polynomials with the same
symmetric Hilbert function decomposition. This is what we do in Section \ref{parampol}
by the use of standard and exotic forms, that explain the unlucky behavior that
the number of variables involved in each homogeneous summand of a given
polynomial $f$ may be larger than what is expected from the Hilbert function
and its symmetric decomposition. Motivated by this, we describe the family ${\mathcal{F}^m_f}$ of polynomials $g$ whose linear partials and Hilbert function coincide with those of $f$, and such that ${g-f}$ is part of an exotic summand for $g$. In Proposition~\ref{blindexoticdescription} we show that the polynomials in ${\mathcal{F}^m_f}$ define isomorphic apolar algebras and compute in Corollary~\ref{blindexoticdimension} its dimension.  In Proposition~\ref{fstandardexotic} we give a decomposition of any polynomial as a sum of a polynomial in standard form and an exotic summand. In Section \ref{apolarschemes} we focus on the local cactus rank by proving Proposition \ref{tail} and computing the local cactus rank of a general cubic surface. Finally in Section \ref{cactusdim} we use this description in order to estimate the dimension of the set of all the polynomials with the same symmetric Hilbert function decomposition, which allows us to estimate the dimension of the Cactus variety. In particular, we prove our lower bound for this dimension in Corollary~\ref{TheCorollary}.

\subsection{Notations}\label{nota}
For main applications  we consider homogeneous forms in $\SV={\KK_{\divp}}[x_0,\ldots ,x_n]$ and their dehomogenization in $S={\KK_{\divp}}[x_1,\ldots ,x_n]$:
\[
\pi_{x_0}:\SV\to S;\quad F(x_0,x_1,...,x_n)\mapsto F(1,x_1,...,x_n).
\] 

We consider the action of the polynomial ring ${\SVdual:= {\KK}[y_0,\ldots ,y_n]}$ on $\SV$ by contraction: if
${\alpha=(\alpha_0,\ldots,\alpha_n)}$ and
${\beta=(\beta_0,\ldots,\beta_n)}$ are multi\h{indices}, then
\[
y^\alpha\big(x^{\beta}\big)=\begin{cases}
x^{\beta-\alpha} &\text{if } \beta\ge\alpha,\\
0 &\text{otherwise.}
\end{cases}
\]
Similarly, we consider the action of the polynomial ring $T:= {\KK}[y_1,\ldots ,y_n]$ on $S$: 
$$T\times S\to S;\quad (\psi, f)\mapsto \psi(f)$$ 
 defined by this contraction restricted to  $y^\alpha,x^\beta$, when  $\alpha_0=\beta_0=0$.

Note that we are using the same notation for ordinary powers in $T$ and divided powers in $S$, unlike what is usually done in the literature (for properties of divided power rings see for instance \cite[Appendix A]{IK}, where the divided power $x^\beta$ would be written as $x^{[\beta]}$).

In characteristic $0$ we could have used ordinary differentiation, and
therefore, by abuse of language, we call ${\psi(f)\in S}$ a \emph{partial} of a polynomial ${f\in S}$ for any ${\psi\in T}$.

\section{Preliminaries}\label{Preliminaries}
\label{polgiv}

We begin this section by presenting the Macaulay correspondence between polynomials and Artinian Gorenstein local rings, which is the starting point of our theory.
\subsection{Macaulay correspondence}

Let $\KK$ be an algebraically closed field of characteristic other than $2,
3$ and consider the divided power
ring ${S:= {\KK}_{\divp}[x_1,\ldots,x_n]}$. Consider the action of the
polynomial ring ${T:= {\KK}[y_1,\ldots ,y_n]}$ on $S$ defined in Subsection \ref{nota}.

Let $S_1$ and $T_1$ be the degree\h{one} parts of $S$ and $T$ respectively. With respect to the action above (classically known as \emph{apolarity}), $S_{1}$ and $T_{1}$ are natural dual spaces and ${\langle x_{1},\ldots,x_{n}\rangle}$ and ${\langle y_{1},\ldots,y_{n}\rangle}$ are dual
bases. The annihilator of a polynomial $f$ of degree $\degF$ is an ideal in  $T$ which we denote by ${f^{\bot}\subset T}$. The quotient ${T_{f}:=T/f^{\bot}}$ is a local Artinian Gorenstein ring (see e.g.\ \cite[Lemma 2.14]{IK}).  In fact, $T_f$  is finitely generated as a \mbox{$\KK$-mod}\-ule  so it is \emph{Artinian}.  The image of $T_1$ in $T_f$ generates the unique maximal ideal $\mathfrak{m}$, so $T_f$ is local.  Furthermore $T_{f}$ has a \mbox{$1$-di}\-men\-sional socle, the annihilator of the maximal ideal, namely $\mathfrak{m}^{\degF}$, so $T_f$ is \emph{Gorenstein}.  In addition, if $f$ is a form, $f^{\bot}$ is a homogeneous ideal and therefore $T_f$ is an Artinian Gorenstein \emph{graded} local ring.

\subsection{Symmetric decomposition of the Hilbert function of a polynomial}
We consider now a polynomial ${f\in S=\KK_{\divp}[x_1,\ldots,x_n]}$ and let ${f^{\bot}\subset T=\KK[y_1,\ldots,y_n]}$ be its annihilator with respect to contraction. We shall interpret a Hilbert function for the local Artinian Gorenstein quotient $T_f=T/f^{\bot}$  in terms of a filtration of the space of partials of the polynomial $f$.  In particular we recall and interpret Iarrobino's
analysis of Hilbert functions on associated graded algebras of $T_f$ and their symmetric decomposition.  We will apply this analysis in the next section to characterize the polynomials with a given Hilbert function.

\medskip

The local Artinian Gorenstein quotient ring ${T_f=T/f^{\bot}}$
is naturally isomorphic to
\[
\df(f)=\{\psi(f)\mid \psi\in T\}
\]
 the space of all partials of $f$, as a $T$-module.

Following \cite{I}, we consider Hilbert functions on graded rings associated to two filtrations of $T_f$. Let $\mathfrak{m}$ be the maximal ideal of $T_f$. The $\mathfrak{m}$-adic filtration
 \[
T_f=\mathfrak{m}^0\supset \mathfrak{m}\supset \mathfrak{m}^2\supset \cdots \supset \mathfrak{m}^\degF\supset \mathfrak{m}^{\degF +1}=0
\]
where $\degF=\deg f$, defines an associated graded ring
\[
T^*_f=\oplus_{i=0}^d \mathfrak{m}^{i}/\mathfrak{m}^{i+1}
\] 
whose Hilbert function  we denote by $H_f$. The L\"oewy filtration
\[
T_f=(0:\mathfrak{m}^{\degF +1})\supset (0:\mathfrak{m}^{\degF})\supset\cdots \supset (0: \mathfrak{m}^2) \supset (0:\mathfrak{m})\supset 0
\]
induces the following sequence of ideals of  $T_f^*$: For each ${a=0,1,2,\ldots}$ let
 \[
C_a=\oplus_{i=0}^{\degF-a}C_{a,i}=\oplus_{i=0}^{\degF-a}\left((0:\mathfrak{m}^{\degF+1-a-i})\cap \mathfrak{m}^{i}\right)/\left((0:\mathfrak{m}^{\degF+1-a-i})\cap \mathfrak{m}^{i+1}\right)\subset T_f^*,
 \]
and consider the $T_f^*$-modules
\[
Q_a=C_a/C_{a+1},\quad a=0,1,2,\ldots
\]
and their respective Hilbert functions
\[
\Delta_{Q,a}=H\big(Q_a\big).
\]
Each $Q_a$ decomposes as a sum  ${Q_a=\oplus_{i=0}^{\degF}Q_{a,i}}$, where ${Q_{a,i}=C_{a,i}/C_{a+1,i}}$. We can check from the definitions that for ${\degF\ge2}$ and any ${a>\degF-2}$, we have ${C_a=Q_a=0}$. The following is an important result on the structure of these modules.
\begin{proposition}\label{symmetry}\cite[Theorem 1.5]{I}  
The $T_f^*$-modules $Q_a,\; a=0,1,2,\ldots ,$  satisfy the following reflexivity condition
 \[
 Q^{\vee}_{a,i}\cong Q_{a,\degF - a-i}.
 \]
In particular, the Hilbert function $ \Delta_{Q,a}=H\big(Q_a\big)$ is symmetric about $(\degF-a)/2$, and thus the Hilbert function $H(T_f^*)$ of $T_f^*$ has a symmetric decomposition
\[
H(T_f^*)=\sum_a \Delta_{Q,a}.
\]
\end{proposition}

The possible symmetric decompositions of the Hilbert function is  restricted by the fact that the partial sums of the symmetric decomposition are Hilbert functions of suitable  quotients of  $T_f^*$.
\begin{corollary}\cite[Section 5B, p.\ 69]{I}\label{quot}
The Hilbert function of $T_f^*/C_{\alpha+1}$ satisfies
\[
H(T_f^*/C_{\alpha+1})=\sum_{a=0}^{\alpha}\Delta_{Q,a}.
\]
In particular every partial sum  $\sum_{a=0}^{\alpha}\Delta_{Q,a}$ is the Hilbert function of a $\KK$-algebra generated in degree $1$.
\end{corollary}

Iarrobino has listed all possible symmetric decompositions of Hilbert functions of rings $T_f^*$ with ${\dim T_f^*\leq 16}$, see \cite[Section 5]{I}.

We now interpret the ideal $C_a$ and the module $Q_a$ in terms of the space $\df(f)$ of partials of $f$.  This interpretation depends on the isomorphism
\[
\tau: T/f^\bot\to \df(f), \quad \psi\mapsto \psi(f),
\]
of $T$-modules and thus {$\KK$-vector} spaces. Let $ \df(f)_{i}$ be the subspace of $\df(f)$ of partials of degree at most $i$. The image of $(0:\mathfrak{m}^i)$ under the map $\tau$ is precisely $\df(f)_{i-1}$, so the L\"{o}ewy filtration
     \[
     (0:\mathfrak{m})\subset(0:\mathfrak{m}^{2})\subset(0:\mathfrak{m}^{3})\subset\cdots \subset(0:\mathfrak{m}^{\degF}) \subset(0:\mathfrak{m}^{\degF+1}) =T_f
     \]
      of $T_f$ is mapped to the degree filtration
      \[
\KK=\df(f)_{0}\subset \df(f)_{1}\subset \df(f)_{2}\subset \cdots \subset \df(f)_{\degF}= \df(f)
\]
 of $\df(f)$. Now  
\[(0:\mathfrak{m}^i)/(0:\mathfrak{m}^{i-1})\cong (\mathfrak{m}^{i-1}/\mathfrak{m}^i)^{\vee},\]
    so the integral function
     \[
    H_f(0)=1, \quad H_{f}(i)=\dim_{\KK}\df(f)_{i}-\dim_{\KK}\df(f)_{i-1}, \quad i=1,\ldots,\degF,
     \]
    coincides with the Hilbert function of $T_f^*$: 
    \[
   H_f(i)=H(T_f^*)(i).
   \]
     On the other hand, the $\mathfrak{m}$-adic filtration
     \[
    T_f \supset \mathfrak{m}\supset \mathfrak{m}^{2}\supset\cdots \supset \mathfrak{m}^{\degF} \supset \mathfrak{m}^{\degF+1} =0
     \]
corresponds to an order filtration on $\df(f)$.
     We call the \emph{order} of $\psi\in T$  the smallest degree of a
     non\h{zero} homogeneous term of $\psi$, and denote it by $\mathrm{ord}(\psi)$.  We call the \emph{order} of 
     a partial $f'$ of $f$ the largest
     order of a ${\psi\in T}$ such that ${f'=\psi (f)}$.   Thus the
 image
 \[
 \tau(\mathfrak{m}^i)\subseteq \df(f)_{\degF-i}
 \]
  is simply the space of partials of order at least $i$ of $f$.

 The  isomorphism  $Q^{\vee}_{a,i}\cong Q_{a,\degF - a-i}$ allows us to interpret the vector space
 $ Q^{\vee}_{a,i}$ as parameterizing partials of $f$ of degree $i$ and order $\degF-a-i$, modulo partials of lower degree and larger order.

 More precisely, let
 $\df(f)_i^a\subset \df(f)$ be the subspace of partials of degree at most $i$ and order at least $\degF-i-a$, then
 \[
  Q^{\vee}_{a,i}\cong \df(f)_i^a/\big(\df(f)_{i-1}^a+\df(f)_i^{a-1}\big).
 \]

So
\begin{equation}\label{valueofdecomposition}
\Delta_{f,a}(i)=\dim_{\KK}\big(Q^{\vee}_{a,i}\big)=\dim_{\KK}\left({\df}(f)_i^a/({\df}(f)_{i-1}^a+{\df}(f)_i^{a-1})\right).
\end{equation}

\begin{notation}
We denote by $\Delta_f$ the symmetric decomposition $H_f=\sum_a \Delta_{f,a}$ of the Hilbert function $H_f$.
\end{notation}

Consider the space of linear forms that are partials of $f$,
\[
{\linearOp(f):=\df(f)_1 \cap S_1},
\]
and the linear subspaces 
\[
    {\linear{f}{a}= \{l\in S_1 \mid l \text{ is a partial of } f
    \text{ of order at least } \degF-a-1 \}} = \df(f)_1^a \cap S_1.
\]
We easily see that for each ${a\ge0}$, we have an isomorphism  
${\linear{f}{a}\simeq {\df}(f)_1^a/{\df}(f)_{0}}$, so
${\Delta_{f,a}(1)=\dim_{\KK} \linear{f}{a} - \dim_{\KK} \linear{f}{a-1}}$. We
obtain a canonical flag of subspaces of $S_1$:
\[
    \linear{f}{0} \subseteq \linear{f}{1} \subseteq  \cdots \subseteq
    \linear{f}{d-2}= \linearOp(f) \subseteq 
    S_1.
\]

\begin{example}\label{partialsofapolynomial}
Let ${f=x_1^{\:3}x_2+x_3^{\:3}+x_4^{\:2}}$. Its space of partials is generated by the elements in the following table, where the generators of each $Q_a^\vee$ are arranged by degree; next to it, we have the symmetric decomposition of its Hilbert function:

\medskip

\begin{center}
\newcommand{\oldarraystretch}{\arraystretch} 
\renewcommand{\arraystretch}{1.2}
\begin{tabular}{c @{\hspace{4em}} c}
    Generators of the space of partials & Hilbert function decomposition \\[2mm]
\begin{tabular}{ l ccccc }
degree & $0$ &$1$ & $2$ & $3$ & $4$ \\
\hline
$Q_0^\vee$ & $1$ & $x_1$,\ $x_2$  & $x_1^{\:2}$,\ $x_1x_2$ & $x_1^{\:3}$,\ $x_1^{\:2}x_2$ & $f$\\
$Q_1^\vee$ & & $x_3$ & $x_3^{\:2}$ & & \\
$Q_2^\vee$ & & $x_4$ & & & \\
& &  &  &  & \\
\end{tabular}
&
$\begin{array}{ccccccc}
    \mbox{degree} && 0 &1 & 2 & 3 & 4 \\
\hline
\Delta_{f,0}& = & 1 & 2 & 2 & 2 & 1 \\
\Delta_{f,1}& = & 0 & 1 & 1 & 0 & 0 \\
\Delta_{f,2}& = & 0 & 1 & 0 & 0 & 0\\
\hline
H_f            & = & 1 & 4 & 3 & 2 & 1\\
\end{array}$
\end{tabular}
\renewcommand{\arraystretch}{\oldarraystretch}
\end{center}
For instance $x_3^{\:2}$ is a partial of order $1$, since it is obtained as ${y_3(f)=x_3^{\:2}}$ and cannot be attained by a higher order element of $T$, so it is a generator of $Q_{1,2}^\vee$. Here we have ${\linear{f}{0}=\langle x_1,x_2\rangle}$, ${\linear{f}{1}=\langle x_1,x_2,x_3\rangle}$, and ${\linear{f}{2}=\langle x_1,x_2,x_3,x_4\rangle}$.
\end{example}

In the next section we shall enumerate polynomials $f$ with a given Hilbert function $H_f$ using this symmetric decomposition $\Delta_f$. For this purpose we denote a Hilbert  function $H$  by its values
\[
   H= \big(H(0),H(1),\ldots,H(\degF)\big)
   \]
   and the decomposition $\Delta_f$,
   \[
   H=\sum_a \Delta_a,
   \]
   by its summands
   \[
       \Delta=(\Delta_0,\ldots,\Delta_{\degF-2}),
   \]
   where each $\Delta_a$ is symmetric around $(\degF-a)/2$, i.e.\ $\Delta_a(i)=\Delta_a(\degF-a-i)$.

  By Corollary \ref{quot}, both $H$ and each partial sum
  \[
  \Delta_{\leq \alpha}=\sum_{a=0}^\alpha \Delta_a
  \]
   are Hilbert functions of  $\KK$-algebras generated in degree $1$, so there are some immediate restrictions on these functions.
  First, Hilbert functions $H$ and   $\Delta_{\leq \alpha}$ have positive
  values and satisfy the Macaulay growth condition
  (cf.\ \cite{M}): If the $i$-binomial expansion of $\Delta_{\leq \alpha}(i)$ is
   \[
\Delta_{\leq \alpha}(i)=\binom{m_i}{i}+\binom{m_{i-1}}{i-1}+\cdots +\binom{m_j}{j}; \quad m_i>m_{i-1}>\cdots>m_j\geq j\geq 1,
   \]
    then
\begin{equation}\label{macaulay}
   \Delta_{\leq \alpha}(i+1)\leq\binom{m_i+1}{i+1}+\binom{m_{i-1}+1}{i}+\cdots+\binom{m_j+1}{j+1}.
 \end{equation}

\begin{example}\label{ref:Hilbdecomposition:example}
For $H(1)=8$, $H(2)\geq 5$ and $1+H(1)+\cdots+H(\degF-1)+ 1=17$ the possible Hilbert functions $H$ and their decompositions $H=\sum_{i}\Delta_{i}$ that satisfy the Macaulay growth conditions are the following:
\[
\begin{array}{cccccc}
H&= &1 &  8&   7 &  1 \\
\Delta_{0}&=   &1 &  7&   7&   1\\
\Delta_{1}&=   &   0  & 1 &  0&   \\
\end{array},
\quad
\begin{array}{ccccccc}
H&= &1 &  8&   6 &  1 &  1\\
\Delta_{0}&=   &1 &  1&   1&  1 &   1\\
\Delta_{1}&=   &   0  & 5 &  5&   0&   \\
\Delta_{2}&=   &   0  & 2&   0&    &
\end{array},
\quad
\begin{array}{cccccccc}
H&= &1 &  8&   5 &  1 &  1 & 1\\
\Delta_{0}&=   &1 &  1&   1&   1&   1& 1\\
\Delta_{1}&=   &   0  & 0 &  0&   0&   0& \\
\Delta_{2}&=   &   0  & 4 &  4&   0&   & \\
\Delta_{3}&=   &   0  & 3&   0&    &  &
\end{array}
\]\[
\begin{array}{ccccccc}
H&= &1 &  8&   5 &  2 &  1\\
\Delta_{0}&=   &1 &  2&   3&   2&   1\\
\Delta_{1}&=   &   0  & 2 & 2&   0&   \\
\Delta_{2}&=   &   0  & 4&   0&   &
\end{array},\quad
\begin{array}{ccccccc}
H&= &1 &  8&   5 &  2 &  1\\
\Delta_{0}&=   &1 &  2&   2&   2&   1\\
\Delta_{1}&=   &   0  & 3 & 3&   0&  \\
\Delta_{2}&=   &   0  & 3&   0&   &
\end{array}.
\]
\end{example}

\section{Standard forms and exotic forms}\label{parampol}

At this point of the analysis we would like to find a precise description of all the polynomials having the same symmetric Hilbert function decomposition. To this purpose  we have firstly to deal with the fact that the number of variables involved in each homogeneous summand of a given polynomial $f$ may be larger than what is expected from the Hilbert function. This will be explained by the appearance of what we will call exotic summands of $f$. We will analyze their role in a description of all polynomials that have a given Hilbert function. 

Let us start with some examples clarifying the kind of phenomena that we have to treat.

\subsection{Standard and ``~Exotic~'' examples}

Let $A$ be a local Artinian Gorenstein algebra. As explained before it can be
represented as a quotient ${A\simeq T/I}$ where $I = f^{\perp}$ for a
polynomial $f\in S$. If the ideal $I$ is fixed, then $f$ is unique up to
action by a unit of $T$, but clearly the choice of $I$ such that ${A\simeq T/I}$ is not unique. In this section we wish to shed some light on how this choice can be made.

\begin{example}\label{ex:embed:part1}
Consider the ring ${A = \KK[\varepsilon]/(\varepsilon^4)}$ and polynomials ${g=x_1^{\:3}}$ and ${h=x_1^{\:3} + x_1x_2}$ in $\KK_\divp[x_1,x_2]$. Then 
\[{A  \simeq  \frac{\KK[y_1,y_2]}{(y_1^{\:4}, y_2)} = \frac{\KK[y_1, y_2]}{g^{\perp}}},\] 
but also 
\[{A  \simeq \frac{\KK[y_1, y_2]}{(y_1^{\:4}, y_2 - y_1^{\:2})} = \frac{\KK[y_1, y_2]}{h^{\perp}}}.\] 
Note that both $x_1$ and $x_2$ occur in $h$, but while $x_1$ is a partial of this polynomial, $x_2$ is not, since its space of partials is ${\df(h)=\langle h, x_1^{\:2} + x_2, x_1, 1\rangle}$.

\smallskip

Now consider the ring ${B = \KK[y_1,y_2]/(y_1^{\:4}-y_2^{\:2}, y_1y_2)}$ and polynomials ${p=x_1^{\:4}+x_2^{\:2}}$ and ${q=x_1^{\:4}+x_1^{\:2}x_2}$ in $\KK_\divp[x_1,x_2]$. Then 
\[{B=\frac{\KK[y_1,y_2]}{p^{\perp}}},\] 
and 
\[{B \simeq \frac{ \KK[y_1, y_2]}{(y_1^{\:3}-y_1y_2, y_2^{\:2})} = \frac{\KK[y_1, y_2]}{q^{\perp}}}.\] 
In this case, $x_2$ is an order\h{one} partial of both $p$ and $q$, since ${y_2(p)=x_2}$ and ${(-y_2+y_1^{\:2})(q)=x_2}$. However, $x_2$ occurs in $q$ in degree $3$, which may be surprising for a linear form that is a partial of order one.

\smallskip

As we will see in the remainder of this section, the most common behaviour is the one we can observe in Example \ref{partialsofapolynomial}: both $x_1$ and $x_2$ are partials of order $3$ and they occur in $f_4$; $x_3$ and $x_4$ are partials of orders $2$ and $1$, and they occur in $f_3$ and $f_2$, respectively.
\end{example}

\subsection{Description of the standard and the ``~exotic~'' phenomena}

Referring the notation of Example  \ref{ex:embed:part1}, we want to distinguish polynomials like $g$ and $p$ that have a ``standard'' behavior  from the ones like $h$ and $q$ where either one finds a variable that does not occur in the partials or a partial whose order does not ``match'' with the degree of the corresponding variable. To this end, in Definition \ref{standard:form} we will define  \emph{standard} forms of polynomials. Intuitively, they correspond to minimal embeddings of algebras, in terms of variables appearing in the related
polynomials $f$. Let ${f\in S}$ and ${A =T/f^{\perp}}$. Moreover let
\begin{equation}\label{homogeneousdecomp}
f=f_\degF+f_{\degF-1}+\cdots+f_0
\end{equation}
be the decomposition in homogeneous summands.

In Section~\ref{polgiv} we defined the Hilbert function of $A$ and its symmetric decomposition $\Delta$. In particular, we saw that ${\Delta_{a}(1) = \dim_{\KK} (\linear{f}{a}/ \linear{f}{a-1})}$, which is space of linear partials of $f$ of order exactly ${\degF-a-1}$. Let
\[
n_i = \sum_{j=0}^{i} \Delta_j(1)=\dim_{\KK} \linear{f}{i},
\]
i.e.\ the dimension of the space of linear partials of $f$ of order at least ${\degF - i - 1}$.
By degree reasons this space is contained in the space of linear partials of $f_{\degF - i} + f_{\degF-i-1} +  \cdots  + f_{\degF}$. But, as we have seen in Example~\ref{ex:embed:part1}, a linear form may occur in $f_{\degF-i}$ and be a partial of $f$ of order less than ${\degF - i - 1}$ or not be a partial of $f$ at all.

First of all let us fix here a basis of linear forms ${x_1, \ldots , x_n}$ in $S_1$ that agrees with the filtration by $\linear{f}{i}$:
\begin{multline}\label{Linfiltration}
\linear{f}{0}=\langle x_1, \ldots , x_{n_0}\rangle \subseteq  
    \linear{f}{1}=\langle x_1, \ldots , x_{n_1} \rangle \subseteq  \cdots\\
  \cdots \subseteq \linear{f}{\degF-2}=\langle x_1, \ldots , x_{n_{\degF-2}} \rangle  \subseteq S_1 = \langle x_1, \ldots , x_n \rangle. 
\end{multline}

\begin{definition}\label{standard:form}
Let ${f\in S}$ be a polynomial with homogeneous decomposition ${f = f_{\degF }+ \cdots + f_0}$. Let $\Delta$ be the symmetric decomposition of the Hilbert function of $T/f^{\perp}$. We say that $f\in S$ is in \emph{standard form} if
\[
    f_{\degF - i} \in \KK_{\divp}\left[\linear{f}{i}\right]=
    \KK_{\divp}\left[x_1, \ldots , x_{n_i} \right], \quad \mbox{for all } i,
\]
where ${x_1, \ldots , x_n}$ is any choice of basis for $S_1$ as in \eqref{Linfiltration}.
\end{definition}
\def\St{\operatorname{StandardForms}}%
We define the linear space of standard
forms.
\[
    \St = \left\{ f\in P_{\leq \degF}\ |\ \forall_i\ f_{\degF - i}
    \in \KK_{\divp}\left[\linear{f}{i}\right] \right\} = \left\{ f\in P_{\leq \degF}\ |\ \forall_i\ f_{i}
    \in \KK_{\divp}\left[\linear{f}{\degF-i}\right] \right\}.
\]

A first important property of standard forms is the following:

\begin{proposition}\label{coex}
The leading summand of a partial of $f$ of degree $\degF-i$ and order $j$ lies
in $\KK_{\divp}[\linear{f}{i-j}]$. 
\end{proposition}
\begin{proof}
Let $g$ be the leading summand of a partial of $f$ of degree $\degF-i$  and
order $j$, then any partial $x$ of degree one of $g$ is a partial of order at
least $\degF-i+j-1$ of $f$ and therefore lies in $\linear{f}{i-j}$. Therefore ${g\in \KK_{\divp}[\linear{f}{i-j}]}$.
\end{proof}

There may be variables appearing in $f$ that do not show up in the leading summands of partials of $f$.   It is tempting to call them exotic variables, but we reserve exotic for the non\h{standard} part of $f$.

\begin{definition}\label{exotic} 
Let ${f = f_{\degF }+ \cdots + f_0\in S}$ be the homogeneous decomposition of $f$, and choose a basis for $S_1$ as in \eqref{Linfiltration}.  The \emph{exotic summand} of degree ${\degF-i}$ of $f$ is the form 
\[
{f_{\degF-i, \infty}\in \langle x_{n_i+1}, \ldots , x_n\rangle\KK_{\divp}[x_1, \ldots , x_n]}
\] 
such that the degree ${\degF-i}$ homogeneous summand of $f$ can be written as
\[
f_{\degF-i}=f_{\degF-i,i}+ f_{\degF-i,\infty},
\]
with $f_{\degF-i,i}\in \KK_{\divp}[x_1, \ldots , x_{n_i}]$.
\end{definition}

Thus $f$ is in standard form if and only if all its exotic summands are zero.

\begin{example}\label{ex:embed:part2}
    Let us see how the above definitions work in the cases of
    Example~\ref{ex:embed:part1}. We have ${H_A = (1, 1, 1, 1)}$ and ${\Delta_0 =
    (1, 1, 1, 1)}$, $\Delta_i = \mathbf{0}$ for $i > 0$. Therefore $n_i = 1$
    for all $i$.

    Now we have ${g=x_1^{\:3}\in \KK_\divp[x_1]}$ so that $x_1^{\:3}$ is in standard form. On
    the other hand, ${x_1x_2\notin \KK_\divp[x_1]}$ so that ${h=x_1^{\:3} + x_1x_2}$ is not in
    standard form. In fact $x_1x_2$ is an exotic summand for $h$. 
    
    For the ring $B$, we get ${H_B = (1, 2, 1, 1, 1)}$, with symmetric decomposition ${\Delta_0 =
    (1, 1, 1, 1, 1)}$, ${\Delta_1 = (0, 0, 0, 0, 0)}$, ${\Delta_2 = (0, 1, 0, 0, 0)}$, 
    $\Delta_i = \mathbf{0}$ for ${i > 2}$. Therefore $n_0 = n_1 = 1$, and ${n_i = 2}$
    for all ${i\ge2}$. As above we can check that $p$ is in standard form, but $q$ is not, 
    having $x_1^{\:2}x_2$ as an exotic summand.
\end{example}

\subsection{Existence of standard forms and their presentation}
\label{standardformspresentation}

Let $A = T/f^{\perp}$ be an Artinian Gorenstein algebra. One could ask:
\begin{enumerate}
    \item if there exists a presentation $A  \simeq T/g^{\perp}$ with $g$ in
        standard form.
    \item in case such a presentation exists, whether there are any relations between $g$ and $f$.
\end{enumerate}

Fortunately, these questions have quite
satisfactory answers, as we explain below.
We need some notation. 
\def\That{\hat{T}}%
\def\ip#1#2{\left\langle #1,\, #2\right\rangle}%
\def\Ddual#1{#1^{\vee}}
\begin{notation}
Let $\That$ denote the power series ring obtained by
completing $T$ at the ideal of the origin. In coordinates, $\That :=
\KK[[y_1, \ldots , y_n]]$. We may interpret $S$ as subset of functionals on
$\That$ via the pairing defined by
\[
    \ip{y^{\alpha}}{x^{\beta}} = \begin{cases}
        1 & \mbox{if}\ \ \alpha = \beta\\
        0 & \mbox{otherwise}.
    \end{cases}
\]
Note that
\begin{equation}\label{eq:innerandcontraction}
\ip{\sigma}{f} = \ip{1}{\sigma(f)},
\end{equation}
as seen by decomposing $\sigma$ and $f$ into monomials.
In particular if $\sigma(f) = 0$, then $\ip{\sigma}{f} = 0$.

Let $\varphi:\That \to \That$ be an automorphism of $\That$. It induces a dual
map $\Ddual{\varphi}: S\to S$ defined by the condition
\begin{equation}\label{pairingadjointrule}
        \ip{\varphi(\sigma)}{f} =
        \ip{\sigma}{\Ddual{\varphi}(f)},\quad\mbox{for all}\quad \sigma\in T,\
        f\in P.
\end{equation}

    Let $I$ be a finite colength ideal of $T$ supported at the origin.     Then $T/I
    = \hat{T}/I$. Clearly, the quotients $\That/I$ and $\That/\varphi(I)$ are isomorphic.
Moreover if $I = f^{\perp}$, then $\varphi(I) =
\big(\Ddual{(\varphi^{-1})}(f)\big)^{\perp}$.
The fundamental result  is that for every $f$ we may find $\varphi$ so that
$\Ddual{\varphi}(f)$ is in standard form.
It fact we prove that $\varphi$ may be chosen ``~with no linear part~''.
\end{notation}
    
 We now make
this precise.
\begin{definition}
    Let $M$ be the unique maximal ideal of $\hat{T}$.
    Let $\varphi: \hat{T}\to \hat{T}$ be an automorphism. We say that
    $\varphi$ is \emph{of order (at least) two} if
    \[\varphi(y_i) = y_i \mod (y_1, \ldots ,y_n)^2\]
    for all $i$.
\end{definition}
\begin{remark}
    Every automorphism of $\hat{T}$ induces a linear action on $M/M^2$.
    The order two automorphisms are precisely those which act trivially. Thus
    they form a normal subgroup. In particular an inverse of an order two
    automorphism is also of order two.
\end{remark}

\begin{theorem}[Existence of standard forms]\label{ref:stform:thm}
    Let ${f\in S}$ be a polynomial with symmetric Hilbert function decomposition $\Delta$.
    Then there is an automorphism ${\varphi:\That\to \That}$ such that $\Ddual{\varphi}(f)$
    is in standard form.
    Consequently, ${f = \Ddual{\psi}(g)}$ for an element
    ${g\in \St}$ and an automorphism ${\psi = \varphi^{-1}}$.

    Moreover one can choose $\varphi$ and $\psi$ of order two.
\end{theorem}
\begin{proof}
    For the existence of $\varphi$ see \cite[Theorem 5.3AB]{I}.
    Take one such $\varphi$. We will compose $\varphi$ with a linear map to
    obtain the required order two automorphism.

    We have $\varphi(y_i) = \sum \lambda_{ij} y_j +
    s_i$, where $s_i\in (y_1, \ldots ,y_n)^2$.
    Let $\tau$ be the \emph{linear} automorphism of $\That$ defined by $\tau(y_i) =
    \sum \lambda_{ij} y_j$. Then $(\varphi \circ \tau^{-1})(y_i) \equiv y_i \mod M^2$ so
    $\varphi \circ \tau^{-1}$ is an automorphism of order two.
    Let ${g = \Ddual{\varphi}(f)}$. By definition $g$ is in standard form.
    Since $\tau$ is a linear automorphism of $\That$, the map
    $\Ddual{(\tau^{-1})}$
    is simply a linear transformation of $S$,~i.e. a change of variables.
    The definition of being in standard form is coordinate free, so that
    $\Ddual{(\tau^{-1})}$ preserves being in standard form. In particular $h =
    \Ddual{(\tau^{-1})}(g)$
    is in standard form.

    But $h = \Ddual{(\varphi \circ \tau^{-1})}(f)$, so that $\varphi \circ
    \tau^{-1}$ is a required automorphism of order two.
\end{proof}

\def\myN{\mathbb{N}_{0}}%
It is also important and interesting to see an explicit description of the
action $\Ddual{\varphi}:S\to S$ of an automorphism $\varphi:\That\to \That$.
For this, recall that $S$ is a \emph{divided power ring}: $x^{\alpha} \cdot
x^{\beta} = \binom{\alpha+\beta}{\alpha}x^{\alpha + \beta}$, where
$\binom{\alpha+\beta}{\alpha} = \prod \binom{\alpha_i+\beta_i}{\alpha_i}$.
\begin{proposition}\label{ref:dualautomorphism:prop}
Let $\varphi: \That\to \That$ be an automorphism. Let $D_i = \varphi(y_i) - y_i\in \That$. For a multi-index $\alpha$ denote $D^{\alpha} = D_1^{\alpha_1} \cdots D_n^{\alpha_n}$. Let $f\in S$. Then
\[
\Ddual{\varphi}(f) = \sum_{\alpha\in \myN^n} x^\alpha\cdot \big(
  D^{\alpha}(f) \big) = f + \sum_{i=1}^n x_i\cdot \big(D_i(f)\big) +  \cdots .
\]
\end{proposition}
\begin{proof}
See~\cite[Proposition~1.8]{J}.
\end{proof}

\begin{example}\label{ex:embed:part3}
Let us illustrate Theorem~\ref{ref:stform:thm} in the setup of Example~\ref{ex:embed:part2}. We should find an $f$ in standard form and an automorphism $\varphi$ of $\KK[y_1,y_2]$ such that ${\Ddual{\varphi}(f) = h= x_1^{\:3} +x_1x_2}$. We see that the linear partials of $h$ are spanned by $x_1$, so if we wish $\varphi$ to be of order at least two, we must have ${f\in\KK_\divp[x_1]}$. According to Proposition~\ref{ref:dualautomorphism:prop},  $\Ddual{\varphi}(f) = \sum_{\alpha\in \myN^n} x^\alpha\cdot \big(    D^{\alpha}(f) \big)$, for some elements of order at least two ${D_1, D_2\in\KK[[y_1,y_2]]}$. Since $f$ must have degree three, ${(D_iD_j)(f)=0}$, for any $i$ and $j$, so
\[x_1^{\:3} +x_1x_2=f+x_1D_1(f)+x_2D_2(f).\]
This implies that ${D_1(f)=0}$ and ${D_2(f)=x_1}$, so we must have ${D_1 = 0 \mod (y_1,y_2)^4}$ and ${D_2=y_1^{\:2} \mod (y_1,y_2)^4}$. Therefore we can choose $\varphi:\That\to\That$ to be the automorphism defined by $\varphi(y_1) = y_1$ and $\varphi(y_2) = y_2 + y_1^{\:2}$.
\end{example}

\subsection{Description of exotic summands}
\label{exoticsummandspresentation}

For parameterization purposes and dimension counts, it is interesting to consider families of polynomials yielding isomorphic algebras, or at least sharing the same Hilbert and symmetric decomposition.  Given a polynomial $f\in \KK_{\divp}[x_1,\ldots,x_{k}]$ such that ${\linearOp(f)=\langle x_1,\ldots,x_k\rangle}$, we consider the family 
\begin{equation}\label{blindexoticfibre}
\begin{aligned}
\mathcal{F}_f^m:=\big\{{g}\in\KK_{\divp}[x_1,\ldots,x_{k+m}]\mid &\; 
    g-f\in(x_{k+1},\ldots,x_{k+m}),\\
  &H_g=H_f,\, \linearOp(g)=\linearOp(f)
    =\langle x_1,\ldots,x_k\rangle\big\}.
\end{aligned}
\end{equation}

The next result gives a characterisation of the elements in this family. We will use the notation $\ls(f)$ for the leading summand of a polynomial $f$, i.e.\ if ${f=f_\degF+\cdots+f_0}$ is its decomposition into homogeneous summands, ${\ls(f)=f_\degF}$.

\begin{proposition}\label{blindexoticdescription}
Let ${f\in\KK_{\divp}[x_1,\ldots,x_k]}$ and assume
\[
{\linearOp(f)=\langle x_1,\ldots,x_k\rangle}.
\]
Let ${g\in\KK_{\divp}[x_1,\ldots,x_{k+m}]}$ be any polynomial. Then ${g\in \mathcal{F}_f^m}$ if and only if there are elements ${\phi_1,\ldots,\phi_m\in\KK[y_1,\ldots,y_k]}$ of order at least two such that
\begin{equation}\label{blindexoticsum}
g=\sum_{i_1,\ldots,i_m\ge0} x_{k+1}^{\ i_1}\cdots x_{k+m}^{\ i_m}\cdot
  \left( \phi_1^{\ i_1}\cdots \phi_m^{\ i_m}\right)(f).
\end{equation}
In particular, for each ${g\in \mathcal{F}_f^m}$, the algebras ${\KK[y_1,\ldots,y_{k+m}]/g^\perp}$ and ${\KK[y_1,\ldots,y_k]/f^\perp}$ are isomorphic.
\end{proposition}
\begin{proof}
Let $l$ be the dimension of ${\df(f)}$ and choose a basis ${h_1,\ldots,h_l}$ for this vector space such that ${\ls(h_1),\ldots,\ls(h_l)}$ are linearly independent. Choose elements ${\psi_1,\ldots,\psi_l\in\KK[y_1,\ldots,y_k]}$ such that for each $i$, ${\psi_i(f)=h_i}$. Let ${g\in \mathcal{F}_f^m}$ and write
\[g=f+x_{k+1}g_1+\cdots+x_{k+m}g_m\]
in such a way that for each $j$, we have ${g_j\in\KK_{\divp}[x_1,\ldots,x_{k+j}]}$. Then for each $i$, 
\[
\psi_i(g)=\psi_i(f)+x_{k+1}\psi_i(g_1)+\cdots+x_{k+m}\psi_i(g_m).
\]
Now Proposition~\ref{coex} tells us that ${\ls\big(\psi_i(g)\big)\in\KK_{\divp}[x_1,\ldots,x_k]}$ and since $\psi_i(f)$ cannot be cancelled by the terms in ${x_{k+1}\psi_i(g_1)+\cdots+x_{k+m}\psi_i(g_m)}$, we must have 
\[
\ls\big(\psi_i(g)\big)=\ls\big(\psi_i(f)\big)=\ls(h_i).
\]
But this implies that ${\psi_1(g),\ldots,\psi_l(g)}$ form a linearly independent set, and since the dimension of $\df(g)$ is also $l$ (since $g$ and $f$ yield algebras with the same Hilbert function), we get 
\[
\df(g)=\langle \psi_1(g),\ldots,\psi_l(g) \rangle. 
\]
In addition, we know that the variables ${x_{k+1},\ldots, x_{k+m}}$ cannot occur in the leading summand of $g$, also by Proposition~\ref{coex}, so the polynomials ${g_1,\ldots,g_m}$ have degree at most ${d-2}$. Now ${y_{k+m}(g)=g_m}$, so $g_m$ is a partial of $g$, which means that there is some ${\phi_m\in\langle \psi_1,\ldots,\psi_l \rangle}$ such that ${g_m=\phi_m(g)}$. Moreover $\phi_m$ has order at least one, because ${\deg g_m\le d-2}$, so ${\phi_m^{\: d+1}(g)=0}$. Denote ${\hat{g}_j=x_{k+1}g_1+\cdots+x_{k+j}g_j}$, and observe that ${y_{k+s}(\hat{g}_j)=g_s}$, if ${s\le j}$, and ${y_{k+s}(\hat{g}_j)=0}$, otherwise. So, 
\begin{align*}
g&=f+\hat{g}_{m-1}+x_{k+m}\cdot\phi_m(g)\\
&=f+\hat{g}_{m-1}+x_{k+m}\cdot\phi_m\big(f+\hat{g}_{m-1}+x_{k+m}\cdot\phi_m(g)\big)\\
&=f+\hat{g}_{m-1}+x_{k+m}\cdot\phi_m(f+\hat{g}_{m-1})+x_{k+m}^{\:2}\cdot\phi_m^{\:2}(g)
\intertext{and iterating this further we get}
g&=\sum_{i\ge0}x_{k+m}^{\:i}\cdot\phi_m^{\:i}(f+\hat{g}_{m-1}).
\end{align*}
Applying $y_{k+m-1}$ to both sides of this equality, we get
\[
y_{k+m-1}(g)=\sum_{i\ge0}x_{k+m}^{\:i}\cdot\phi_m^{\:i}\big(y_{k+m-1}(f+\hat{g}_{m-1})\big)
  =\sum_{i\ge0}x_{k+m}^{\:i}\cdot\phi_m^{\:i}(g_{m-1}).
\]
Again there must be some ${\phi_{m-1}\in\langle \psi_1,\ldots,\psi_l \rangle}$ such that ${\sum_{i\ge0}x_{k+m}^{\:i}\cdot\phi_m^{\:i}(g_{m-1})=\phi_{m-1}(g)}$. So,
\begin{align*}
g&=\sum_{i\ge0}x_{k+m}^{\:i}\cdot\phi_m^{\:i}(f+\hat{g}_{m-2}+x_{k+m-1}g_{m-1})\\
&=\sum_{i\ge0}x_{k+m}^{\:i}\cdot\phi_m^{\:i}(f+\hat{g}_{m-2})
  +x_{k+m-1}\sum_{i\ge0}x_{k+m}^{\:i}\cdot\phi_m^{\:i}(g_{m-1})\\
&=\sum_{i\ge0}x_{k+m}^{\:i}\cdot\phi_m^{\:i}(f+\hat{g}_{m-2})+x_{k+m-1}\cdot\phi_{m-1}(g).
\intertext{Iterating we get}
g&=\sum_{i_{m-1},i_m\ge0}x_{k+m-1}^{\:i_{m-1}}x_{k+m}^{\:i_m}\cdot
  \phi_{m-1}^{\:i_{m-1}}\phi_m^{\:i_m}(f+\hat{g}_{m-2}).
\end{align*}
Applying the remaining operators ${y_{k+m-2},\ldots,y_{k+1}}$ the same way, we obtain \eqref{blindexoticsum}. It remains to show that ${\phi_1,\ldots,\phi_m}$ have order at least two. Without loss of generality, suppose that $\phi_1$ has order one, and write ${\phi_1=\phi'+\phi''}$, where ${\phi'\in\langle y_1,\ldots,y_k\rangle}$ and ${\ord \phi''\ge2}$. Let ${t\in\langle x_1,\ldots,x_k\rangle}$ be such that ${\phi'(t)=1}$. Since $t$ is a partial of $f$, there exists ${\eta\in\KK[y_1,\ldots,y_k]}$ such that ${\eta(f)=t}$. Note that ${(\phi_u\eta)(f)=\phi_u(t)}$ is a constant, for all ${u>1}$. So
\[\eta(g)=\sum_{i_1,\ldots,i_m\ge0} x_{k+1}^{\ i_1}\cdots x_{k+m}^{\ i_m}\cdot
  \phi_1^{\ i_1}\cdots \phi_m^{\ i_m}\eta(f)=t+x_{k+1}+\sum_{u>1}x_{k+u}\phi_u(t),\]
and therefore $\eta(g)$ is a partial of $g$ of degree one that does not belong to ${\langle 1, x_1,\ldots,x_k\rangle}$, a contradiction. 

For the converse, suppose that $g$ admits a presentation as in \eqref{blindexoticsum}. Then clearly we have ${g-f\in(x_{k+1},\ldots,x_{k+m})}$. Let $\varphi$ be the automorphism of ${\KK[y_1,\ldots,y_{k+m}]}$ defined by ${\varphi(y_i)=y_i}$ for ${1\le i\le k}$ and ${\varphi(y_{k+i})=y_{k+i}+\phi_i}$, for ${1\le i\le m}$. Then by Proposition~\ref{ref:dualautomorphism:prop}, we see that ${g=\Ddual{\varphi}(f)}$. So $\varphi$ induces an isomorphism between ${\KK[y_1,\ldots,y_{k+m}]/g^\perp}$ and ${\KK[y_1,\ldots,y_k]/f^\perp}$, which proves the last statement, and shows that ${H_g=H_f}$. Finally, if ${t\in\linearOp(f)}$ and we take ${\eta\in\KK[y_1,\ldots,y_k]}$ such that ${\eta(f)=t}$, we may apply $\eta$ to both sides of \eqref{blindexoticsum} to get ${\eta(g)=t}$. So ${\linearOp(g)\subseteq\linearOp(f)}$ and since they have the same dimension, equality must hold, and ${g\in \mathcal{F}_f^m}$.
\end{proof}

\begin{remark}
We can get an alternative proof of Proposition~\ref{blindexoticdescription} if we take ${g\in \mathcal{F}_f^m}$ and use Proposition~\ref{coex} to see that since the linear partials of $g$ lie in ${\langle x_1,\ldots,x_k\rangle}$, the leading terms of partials of $g$ only involve these variables, like we did in the proof. Now if we denote $(g^\perp)^*$ the ideal generated by the initial forms of elements of $g^\perp$, we know that $(g^\perp)^*$ is the annihilator of the set of leading summands of all partials of $g$ (see a discussion on this at the beginning of Section~2 in \cite{CN13}, but also Proposition~3 in \cite{E}). Therefore for each ${j>k}$, ${y_j\in(g^\perp)^*}$, which means that there is an element ${D_j\in T}$ of order at least two such that ${y_j-D_j\in g^\perp}$. If a variable $y_{j_1}$ occurs in a $D_{j_2}$, with ${j_1,j_2>k}$ then we may replace $y_{j_1}$ by $D_{j_1}$ in $D_{j_2}$. Each time we do this, the minimal degree of $y_{j_1}$ in $D_{j_2}$ grows. When this degree exceeds $\deg g$, we may discard the remaining part, so the process eventually ends and we may assume that each ${D_j \in k[y_1, \ldots , y_k]}$. We can then consider the  automorphism ${\varphi:\That \to \That}$ which sends $y_j$ to itself for ${j \le k}$ and $y_j$ to ${y_j - D_j}$ for ${j > k}$. We can now show that ${\Ddual{\varphi}(g) = f}$ by showing that for any element ${\phi\in T}$, ${\big\langle \phi, \Ddual{\varphi}(g) \big\rangle=\langle \phi, f\rangle}$.
\end{remark}

Note that in the definition of the family ${\mathcal{F}_f^m}$ and the hypotheses of Proposition~\ref{blindexoticdescription}, the polynomial $f$ need not be in standard form. This result gives us a way of adding exotic summands to a polynomial without changing the Hilbert function of the algebra it yields. It also gives us the following result:
\begin{corollary}\label{blindexoticdimension}
Let ${f\in\KK_{\divp}[x_1,\ldots,x_{k}]\subset \KK_{\divp}[x_1,\ldots,x_{k+m}]}$ be a polynomial of degree $\degF$ and assume
$
{\linearOp(f)=\langle x_1,\ldots,x_k\rangle}.
$
Then the family $\mathcal{F}_f^m\subset \KK_{\divp}[x_1,\ldots,x_{k+m}]_{\leq d}$ has dimension
\[\dim \mathcal{F}_f^m=m\cdot\sum_{\substack{a\ge0\\ i\le\degF-a-2}} \Delta_{f,a}(i) \]
\end{corollary}
The next result shows that all exotic summands may be obtained in a similar description.

\begin{proposition}\label{fstandardexotic}
Let ${f\in S}$ be a polynomial of degree $d$ and choose a basis ${x_1,\ldots,x_n}$ for $S_1$ that agrees with the filtration in \eqref{Linfiltration}. Then we can write ${f=f_{\mathrm{st}}+f_{\mathrm{ex}}}$ such that $f_{\mathrm{st}}$ is in standard form and 
\begin{equation}\label{fexotic}
f_{\mathrm{ex}} = \sum_{\substack{
  \alpha>0\\
  \alpha_1=\cdots=\alpha_{n_1}=0}} x^\alpha\cdot \big(D^{\alpha}(f_{\mathrm{st}}) \big).
\end{equation}
where ${D_{n_1+1},\ldots,D_n\in T}$ have order at least two and 
\begin{enumerate}
\item ${\degF-a\le \deg D_k(f_{\mathrm{st}})\le\degF-2}$, if ${D_k(f_{\mathrm{st}})\ne0}$ and ${n_{a-1}<k\le n_a}$,  for any ${a\ge2}$;
\item ${\deg D_k(f_{\mathrm{st}}) \le\degF-2}$, if ${k> n_{\degF-2}}$.
\end{enumerate}
\end{proposition}
\begin{proof}
Applying Theorem \ref{ref:stform:thm} and Proposition~\ref{ref:dualautomorphism:prop} to $f$, we know that there exists a polynomial $g$ in standard form such that  
\begin{equation}\label{fstandardexoticproofdisplay} 
f=\Ddual{\varphi}(g) = \sum_{\alpha\in \myN^n} x^\alpha\cdot D^{\alpha}(g),
\end{equation}
with ${D_1,\ldots,D_n\in T}$ of order at least two. 

\emph{Claim.} The basis ${x_1,\ldots,x_n}$ also satisfies the filtration in \eqref{Linfiltration} for the polynomial $g$. To see this, let ${t\in\langle x_1,\ldots,x_n\rangle}$ be a linear partial of $g$ of order $j$, and ${\eta\in T}$ an element of order $j$ such that ${\eta(g)=t}$. If ${\langle\ ,\ \rangle:T\times S}$ is the usual pairing, using the rule in \eqref{pairingadjointrule}, we have 
\[
\big\langle\phi,\big(\varphi^{-1}(\eta)\big)(f)\big\rangle
  =\big\langle\big(\varphi^{-1}(\eta)\big)\cdot\phi,f\big\rangle=\langle\varphi(\phi),t\rangle,
\]
for any ${\phi\in T}$. Since $\varphi$ is of order at least two, this implies ${\big(\varphi^{-1}(\eta)\big)(f)=t}$. We also know that ${\ord\varphi^{-1}(\eta)=\ord\eta}$, so $t$ is also a linear partial of $f$ of order $j$. So $f$ and $g$ have the same linear partials and those partials have the same order, which proves the claim. 

Now we wish to show that $g$ can be replaced by a polynomial $g'$, also in standard form, so that we do not need the first operators ${D_1,\ldots,D_{n_1}}$. We consider expression \eqref{fstandardexoticproofdisplay} , and we start by factoring the power of any variable, say $x_1^{\:\alpha_1}$, and the power of the corresponding operator, $D_1^{\:\alpha_1}$. Note that $D_1^{\:\alpha_1}$ obviously commutes with the remaining operators ${D_{2}^{\:\alpha_{2}} \cdots D_n^{\:\alpha_n}}$. In order to move also the term $x_1^{\:\alpha_1}$ to the right side of ${D_{2}^{\:\alpha_{2}} \cdots D_n^{\:\alpha_n}}$, we can apply the rule
\[
x_i\big(\phi(h)\big)=\phi(x_ih)-\phi^{(i)}(h), \quad \text{for any } \phi\in T \text{ and } h\in S,
\]
where ${\phi^{(i)}=\tfrac{\partial}{\partial y_i}\phi}$. We obtain
\begin{align*}
f &= \sum_{\alpha\in \myN^n} (x_2^{\:\alpha_2}\cdots x_n^{\:\alpha_n})x_1^{\:\alpha_1}  
    \cdot \big( ( D_{2}^{\:\alpha_{2}} \cdots D_n^{\:\alpha_n} )D_1^{\:\alpha_1}(g) \big),\\ 
  &= \sum_{i\ge0}\sum_{\alpha_2,\ldots,\alpha_n} (x_2^{\:\alpha_2}\cdots x_n^{\:\alpha_n})
    \cdot \Big( \big(D_{2}-D_{2}^{\:(1)}D_1\big)^{\:\alpha_{2}} \cdots 
    \big(D_{n}-D_{n}^{\:(1)}D_1\big)^{\:\alpha_n}\Big) 
    \big(x_1^{\:i} \cdot D_1^{\:i}(g)\big),\\
  &= \sum_{\alpha_2,\ldots,\alpha_n} (x_2^{\:\alpha_2}\cdots x_n^{\:\alpha_n})
    \cdot \big( (D'_{2})^{\:\alpha_{2}} \cdots(D'_{n})^{\:\alpha_n}\big) (g'),
\end{align*}
where ${g'=\sum_{i\ge0}x_1^{\:\alpha_1} \cdot D_1^{\:\alpha_1}(g)}$ and ${D'_i=D_{i}-D_{i}^{\:(1)}D_1}$. The equality between the first and the second lines can be checked by hand in a straightforward, even if cumbersome, computation. Observe that in this way we get an extra piece $D_{i}^{(1)}$, but this does not change the fact that the new elements ${D'_1,\ldots,D'_n}$ have order at least two. Repeating this procedure, we can rewrite
\begin{equation*}
f = \sum_{\substack{
  \alpha\in \myN^n\\
  \alpha_1=\cdots=\alpha_{n_1}=0}} x^\alpha\cdot \big(
  D''^{\alpha}(g'') \big),\qquad \text{where} \qquad 
  g'' = \sum_{\alpha\in \myN^{n_1}} x^\alpha\cdot \big(
  (D'')^{\alpha}(g) \big),
\end{equation*}
with some modified ${D''_1,\ldots,D''_n\in T}$ of order at least two. Since any terms of $f$ and any terms of $g$ involving only the variables ${x_1,\ldots,x_{n_1}}$ are not exotic (because the choice of basis ${x_1,\ldots,x_n}$ agrees with the filtration in \eqref{Linfiltration} for both $f$ and $g$) we have that $g''$ is in standard form.

Now we know that for any ${k>n_1}$, ${D''_k(g'')}$ is a partial of degree at most ${\degF-2}$, otherwise ${x_k}$ would occur in the leading summand of $f$, and all such linear partials belong to ${\langle x_1,\ldots,x_{n_0}\rangle}$. Finally, suppose that for some $k$ we have ${n_{a-1}<k\le n_a}$ but ${\deg D''_k(g'')<\degF-a}$. Then the term ${x_kD''_k(g'')}$ has degree at most ${d-a}$ and belongs to ${\KK_\divp[x_1,\ldots,x_{n_a}]}$, so it is not part of an exotic summand. The terms ${x_k^{\:2}{D''_k}^2(g''), \dots, x_k^{\:\degF-1}{D''_k}^{\degF-1}(g'')}$ have lower degree so  are also not part of exotic summands. Therefore we can perform another modification as above, with the variable $x_k$ and the corresponding operator $D''_k$. Doing this for every such $k$, we may replace $g''$ by some $g'''$ also in standard form, and $f$ will be written as in \eqref{fexotic}, with ${f_{\mathrm{st}}=g'''}$.
\end{proof}

\begin{example}\label{ex:11111}
    Let $f$ be a polynomial of degree five such that $H_{f} = (1, 1, 1, 1,
    1, 1)$; this is the minimal possible Hilbert function. Then $H_f$ is symmetric, so that
    the only possible symmetric decomposition is $\Delta_{f, 0} = H_f = (1, 1, 1,
    1, 1, 1)$ and all other $\Delta_{f, i}$ equal to zero vectors.

    By Theorem~\ref{ref:stform:thm} we see that $f = \Ddual{\varphi}(g)$,
    where $g$ is in standard form and $\varphi$ is of order at least two. We have $\Delta_{g, i} =
    \Delta_{f, i}$ for all $i$, so $\dim_{\KK} \linear{g}{i} = 1$ for
    all $i$. Choose a generator $x$ for this space.
    Let $g = g_5 +  \cdots  + g_0$ be the decomposition into homogeneous
    summands. From the definition of standard form we see that $g_i\in \KK[x]$
    for all $i$, i.e. $g\in \KK[x]$.
    So we may write $g = a_5x^5 + a_4 x^4 + \cdots  + a_0$ for constants
    $a_i$. Since $g$ has degree five, we have $a_5\neq 0$ and by changing $x$
    we may assume that $a_5 = 1$.

    Now $f = \sum_{\alpha\in \myN^n} x^\alpha\cdot \big(
        D^{\alpha}(g) \big)$,
        where $D_i = \varphi(y_i) - y_i\in M^2$.
        Then $D_iD_jD_k(g) = 0$ for all $i,j,k$ and the sum becomes shorter:
        $f = g + \sum_{i} x_i\cdot D_i(g) + \sum_{i,j} x_{i}x_j\cdot
            D_iD_j(g)$.
        We see that $\deg\big(D_i(g)\big) \leq 3$ so that
        \[f = x^5 + \sum_{i=1}^n \lambda_ix_i
        \cdot  x^3 + \sum_{i=1}^{n} \mu_i x_i  x^2 + \sum_{i, j}
        \lambda_i \lambda_j x_ix_jx + Q,\]
        where $\lambda_i, \mu_i$ are constants and $Q$ is a polynomial of
        degree at most two, partially depending on $\lambda_i$ and $\mu_i$.
        What is the dimension of possible $f$ obtained this way? Each
            $D_i$ may be chosen as an element of the square of the maximal
            ideal of $T/(x^5)^{\perp}$, therefore
            we have $(5-2)$-dimensional choice. Together with the choice of $x$,
        we obtain at most a $4(\dim \Spec(T) )= 4n$-dimensional family.
\end{example}

\section{Apolarity and local cactus rank}\label{apolarschemes}
We shall now apply our analysis of Hilbert functions of  polynomials to apolar subschemes of a homogeneous form.  Recall from \ref{nota} that we denote by $\SV={\KK_{\divp}}[x_0,\ldots ,x_n]$ and $\SVdual={\KK}[y_0,\ldots ,y_n]$.

\begin{definition}\label{apolardef}
A subscheme ${Z\subset {\mathbb P}(\SV_{1})}$ is apolar to a form ${F\in \SV}$ if its homogeneous ideal ${I_{Z}\subset \SVdual}$ is contained in $F^{\bot}$.
\end{definition}
Apolarity for a subscheme  ${Z\subset {\mathbb P}(\SV_{1})}$ to a form $F$ of degree $\degF$ may be given the following natural interpretation in terms of the $d$-uple embedding of $Z$, the image $\nu_{\degF}(Z)\subset {\mathbb P}(\SV_{d})$ where ${\nu_\degF: \PP(\SV_1)\to  \PP(\SV_\degF)}$,  
${[l]\mapsto [l^\degF]}$.
\begin{lemma}\label{apolarlemma} (Apolarity Lemma). A scheme $Z\subset {\mathbb P}(\SV_{1})$ is apolar to $F\in \SV_d$ if and only if  $[F]\in \langle \nu_{\degF}(Z) \rangle\subset {\mathbb P}(\SV_{d})$.
\end{lemma}
\begin{proof}
If $Z$ is apolar to $F$, then ${(I_Z)_\degF\subseteq (F^{\bot})_\degF}$ and we get ${[F]\in V\big((I_Z)_\degF\big)=\langle \nu_{\degF}(Z) \rangle \subset \PP(\SV_\degF)}$, so the ``~only if~''  part follows.  For the ``~if~'' part,  $(I_Z)_e\subset (F^{\bot})_e=\SVdual_e$ when $e>\degF$, so it remains to consider ${\Psi\in (I_Z)_e}$, for some ${e\le d}$.  In this case, ${\SVdual_{d-e}\Psi\subset (I_Z)_\degF}$. So if  ${[F]\in \langle \nu_{\degF}(Z) \rangle}$, then ${\SVdual_{d-e}\Psi\subset (F^{\bot})_\degF}$. But ${\SVdual_{d-e}\Psi( F)=0}$ only if $\Psi(F)=0$, so the ``~if'' part follows also.
\end{proof}

We are particularly interested in minimal apolar zero\h{dimensional} subschemes to a form, their length is called the cactus rank of the form.
The closure of the  set of forms with a given cactus rank  is called a cactus variety of forms, although it may be reducible.
Minimal apolar zero\h{dimensional} schemes are locally Gorenstein, so our first aim is to describe the forms that have a given minimal length local Gorenstein scheme.
 For any form $F$ a minimal apolar zero\h{dimensional} scheme decomposes into local Artinian Gorenstein schemes. This decomposition corresponds to an additive decomposition of the form $F$.  In particular, the cactus variety is the join of varieties of forms whose minimal apolar scheme is local.

    Let $F$ be a form of degree $\degF$ in $\SV$ and $\ptk\in \PP(\SV_1)$.
Consider the family ${\mathcal{Z}\subset \textrm{Hilb}\big(\mathbb{P}(\SV_1)\big)}$ of subschemes apolar to $F$ and
    supported at $\ptk$.
    We construct a particular subscheme ${\Zfl\in
    \mathcal{Z}}$ that we will call the natural apolar subscheme of $F$ at $\ptk$.
        The element ${\ptk\in \PP(\SV_1)}$ defines a hyperplane ${V(l) \subset
    \PP(\SVdual_1)}$.
    The complement $U$ of this hyperplane is
    isomorphic to an affine space $\Spec (\SW)$.
    Moreover, we get a homomorphism ${\pi = \pi_{[l]}:\SV\to \SW}$
    corresponding to passing from the homogeneous coordinate ring $S$ of
    $\PP(\SVdual_1)$ to the coordinate ring of $U$.

Choose dual bases ${\langle l,l_{1},\ldots,l_{n}\rangle}$ and ${\langle l',l'_{1},\ldots,l'_{n}\rangle}$ for $\SV_1$ and $\SVdual_1$, respectively. Let $\SW=\KK[l_{1},\ldots,l_{n}]$ and let $\SWdual=\KK[l'_{1},\ldots,l'_{n}]$. Then $(\SW)_1$ and $(\SWdual)_1$  are natural dual spaces like $\SV_1$ and $\SVdual_1$ above, and $\SWdual$ is the coordinate ring of the affine space ${U'}$ that contains the point $[l]$ and is the complement of the hyperplane ${V(l') \subset\PP(\SV_1)}$. Given any polynomial ${g\in\SW}$ we denote by $Z_g$ the subscheme ${V(g^{\perp}) \subseteq U'\subset \mathbb{P}(\SV_1)}$.

    \begin{definition}\label{ref:naturalapolar:def}
        Let $F\in S$ be any form and $l\in \SV_1$ a  linear form.
        Take $f := \pi_{[l]}(F)$ and $f^{\perp} \subseteq \SWdual$.
        We define $\Zfl$ to be the subscheme ${Z_f=V(f^{\perp}) \subseteq U'\subset \mathbb{P}(\SV_1)}$.
    \end{definition}

        Since $\Zfl$ is finite, it is a closed subset of $\PP(\SV_1)$. By construction, the support of $\Zfl$ is $[l]\in \PP(\SV_1)$.

    In the definition of $\Zfl$ we have not used any particular coordinate system.
    However, to simplify the following proofs we fix coordinates.
    First, note that for every lifting ${l\in \SV_1}$ of
    $[l]$ we have a canonical isomorphism ${\SW  \simeq \SV/(l-1)}$.
    Changing ${x_0, \ldots, x_n}$ if necessary we may assume that ${l = x_0}$.
    Then $\SW$ may be identified with $S=\KK_{\divp}[x_1, \ldots ,x_n]$ and $\SWdual$ with
    $T=\KK[y_1, \ldots ,y_n]$. The homomorphism
    \[\pi:=\pi_{x_0}:\SV\to S\]
    sends $x_0$ to $1$ and $x_i$ to itself for $i > 0$. Note that $\pi$ induces an
    isomorphism between the space $\SV_{\degF}$ of homogeneous polynomials of
    degree $\degF$ and the space $(S)_{\leq \degF}$ of polynomials of
    degree $\degF$.
    Furthermore we fix a ``dual'' homomorphism
    \[\pi^*:\SVdual \to T\]
    sending $y_0$ to $1$ and $y_i$ to itself for $i > 0$.

Let ${G\in \SV}$ be a homogeneous polynomial, let ${\Psi \in \SVdual}$ be a homogeneous differential operator, and denote ${g = \pi(G)}$ and ${\psi = \pi^*(\Psi)}$. In general, ${\pi\big(\Psi(G)\big)}$ and $\psi(g)$ are different polynomials, but the following lemma gives the basic relation between them.

\begin{lemma}\label{ref:localtoglobal:lem}
        Let ${G\in \SV}$ be a homogeneous polynomial, let
        ${\Psi\in \SVdual}$ be a homogeneous operator, with ${\deg G\ge\deg\Psi}$, and let ${\degF = \deg G-\deg\Psi}$.
        Let ${g = \pi(G)}$ and ${\psi = \pi^*(\Psi)}$, then the degree-$\degF$ tails
        of ${\pi\big(\Psi(G)\big)}$ and $\psi(g)$ are equal.
        Moreover, if $G$ is divisible by ${x_0}^{\deg\Psi}$, then ${\pi\big(\Psi(G)\big)
        = \psi(g)}$.
\end{lemma}
\begin{proof}
        The first statement is equivalent to saying that the images of
        $\pi\big(\Psi(G)\big)$ and
        $\psi(g)$ are equal in the linear space $S/(S)_{>\degF}$. Therefore,
        both statements are linear with respect to $\Psi$ and $G$ and it is
        enough to prove them in the case when $G$ and $\Psi$ are
        monomials. Let $G = x_0^{\alpha_0} \ldots x_{n}^{\alpha_n}$ and $\Psi =
        y_0^{\beta_0} \ldots  y_n^{\beta_n}$.
        By definition
        \begin{align*}
            \Psi(G)&=\begin{cases}
                x_0^{\alpha_0 - \beta_0}x_1^{\alpha_1 - \beta_1}  \cdots
                  x_n^{\alpha_n - \beta_n}
                &\text{if } \alpha_i\geq\beta_i\mbox{ for all } i\geq 0,\\
                0 &\text{otherwise,}
        \end{cases}\\
        \psi(g)&=\begin{cases}
                x_1^{\alpha_1 - \beta_1} \cdots x_n^{\alpha_n - \beta_n}
                &\text{if } \alpha_i\geq\beta_i\mbox{ for all } i\geq 1,\\
                0 &\text{otherwise.}
            \end{cases}
        \end{align*}
        We consider two cases. First, suppose $\alpha_0 \geq \beta_0$.
        Then the conditions ${\forall{i\geq 0},\ \alpha_i - \beta_i \geq
        0}$ and ${\forall{i\geq 1},\ \alpha_i - \beta_i \geq 0}$ are equivalent. Thus 
        ${\pi\big(\Psi(G)\big) = \psi(g)}$ and so their images in $S/(S)_{>\degF}$ agree.

        Next, suppose that ${\alpha_0 < \beta_0}$.
        Then $\Psi(G) = 0$. Suppose $\psi(g) \neq 0$ in $S$. Then $\psi(g) =
        x_1^{\alpha_1 - \beta_1}  \ldots  x_n^{\alpha_n - \beta_n}$ is
        a monomial of degree $\sum_{i\geq 1} \alpha_i - \beta_i = d - (\alpha_0 -
        \beta_0) > \degF$, thus its image in $S/(S)_{>\degF}$ is zero, i.e.\
        equal to the image of $\Psi(G)$. This finishes the proof of the first
        claim.

        For the second claim, note that by assumption ${\alpha_0 \geq \deg
        \Psi\geq \beta_0}$, thus the proof of the first case above applies,
        giving $\pi(\Psi(G)) = \psi(g)$.
\end{proof}

    \begin{corollary}\label{ref:localzeromeansglobalzero:cor}
        Let ${G\in \SV}$ and ${\Psi\in \SVdual}$ be homogeneous polynomials.  Let ${g =
        \pi(G)}$ and\linebreak ${\psi = \pi^*(\Psi)}$. If
        ${\psi(g) = 0}$, then ${\Psi(G) = 0}$.
    \end{corollary}

    \begin{proof}
        Let ${b = \deg(\Psi)}$ and ${G' := x_0^{b}G}$. Then ${g = \pi(G')}$ and $G'$,
        $\Psi$, $g$, $\psi$ satisfy assumptions of
        Lemma~\ref{ref:localtoglobal:lem}, therefore ${\pi\big(\Psi(G')\big) = \psi(g) = 0}$.
        Since $\pi$ is an isomorphism between $\SV_{\degF}$ and
        $(S)_{\leq \degF}$ we have ${\Psi(G') = 0}$. But ${G =
        y_0^{b} (G')}$, so that
        \[{\Psi(G) = \Psi\big(y_{0}^b(G')\big) =y_0^{b}\Psi(G') = 0}.\qedhere\]
    \end{proof}

    \begin{corollary}\label{ref:containment:lem}
        Let ${F\in S}$ be a homogeneous polynomial and ${l\in \SV_1}$ any linear form.  Then the scheme $\Zfl$  (see Definition~\ref{ref:naturalapolar:def}) is apolar to $F$.
    \end{corollary}
    \begin{proof}
        Take any homogeneous form $\Psi\in I(\Zfl)$. Then $\psi=\pi^*(\Psi)$ is an
        element annihilating $f=\pi(F)$.
        By Corollary~\ref{ref:localzeromeansglobalzero:cor} we have $\Psi(F) =
        0$, i.e.\ ${\Psi\in F^{\perp}}$.
    \end{proof}

\begin{remark}
    It is easy to characterize the $d$-uple embedding $\nu_\degF(\Zfl)$ in a manner similar to the proof
    of the Apolarity Lemma \ref{apolarlemma}.

            For $l\in \SV_1$, let $l^{\perp}\cap \SVdual_1$ be the subspace
            of linear forms in $\SVdual_1$ that annihilate $l$, and let
            $(l^{\perp})^e$ be its $e$-th symmetric product.  Then
            ${(l^{\perp})^e(F)\subset \SV_{\degF-e}}$ for ${e\le\degF}$ is a subspace of
            partials of $F$ or degree $\degF-e$, and the linear span of
            $\nu_\degF(\Zfl)$ is given by
\[
    \langle \nu_{\degF}(\Zfl)\rangle=\PP\big( l^\degF\oplus l^{\degF-1}\cdot
  (l^{\perp})^{\degF-1}(F)\oplus \cdots \oplus  l\cdot (l^{\perp})^{1}(F)\oplus F\big).
\]
Furthermore ${\nu_{\degF}(\Zfl)= \langle \nu_{\degF}(\Zfl)\rangle\cap \Ver{\degF,n}\subset\PP(\SV_\degF)}$.
\end{remark}

The following lemma is a private communication from Jaros\l{}aw Buczy\'nski. 
    \begin{lemma}[Buczy\'nski]\label{ref:reduction:lem}
        If $Z$ is any local scheme in $\PP(\SV_1)$ apolar to a homogeneous polynomial ${F\in S}$
        and supported at
        $[l] = [1:0:\cdots : 0]$, then there exists a closed subscheme ${Z' \subseteq Z}$, apolar to $F$,
        such that $Z'= Z_{G, l}$, for some ${G\in \SV}$. Moreover,
        $F = \Psi(G)$ for some $\Psi\in \SVdual$.
    \end{lemma}

    \begin{proof}
        \def\pointid{\mathfrak{p}}
        By \cite[Proposition 2.2, Lemma 2.3]{BB},
        the scheme $Z$ contains a
        closed Gorenstein subscheme $Z'$ apolar to $F$.

        \def\ZGl{Z_{G, l}}
        Let ${g\in S}$ be a polynomial such that $Z' = V(g^{\perp})$
        and let $G\in \SV$ be a  homogenization of $g$ such that $G$ is
        divisible by $x_0^\degF$, where $\degF = \deg(F)$.
        Then ${g = \pi(G)}$ and ${Z' = Z_{G, l}}$. Lemma~\ref{ref:localtoglobal:lem}
        asserts that ${(G^{\perp})_s = (I_{\ZGl})_s}$ for any ${s\le\degF}$, and therefore ${(G^{\perp})_s = (I_{\ZGl})_s
        \subseteq (F^{\perp})_s}$.
        Since $F$ is of degree $\degF$, it follows that
        $G^{\perp} \subseteq F^{\perp}$, so $F$ is a partial of $G$.
    \end{proof}

We may now prove Proposition~\ref{tail}.

\begin{proof}[Proof (Proposition \ref{tail})]
By Lemma \ref{ref:reduction:lem}, we may assume that $\Gamma=Z_{G,
    l}$ and that $F=\Psi(G)$ for some homogeneous $\Psi\in \KK[y_0,\ldots,y_n]$.    Let
    $g=\pi(G)$ and $\psi=\pi^*(\Psi)$. By
    Lemma~\ref{ref:localtoglobal:lem} 
    the polynomial $f$ is the degree-$\degF$ tail of $\psi(g)$.

    Clearly ${\df\big(\psi(g)\big)\subseteq\df(g)}$, thus $Z_{\psi(g)}$ is a
    closed subscheme of $Z_{g} = \Gamma$.
By minimality of $\Gamma$ it is enough to prove that
$Z_{\psi(g)}$ is a scheme apolar to $F$.
Let $\psi'$ be such that ${\psi'\big(\psi(g)\big) = 0}$ and let ${\Psi'\in \SVdual}$ be a
homogeneous polynomial such that ${\psi' = \pi^*(\Psi')}$. Then ${\pi^*(\Psi'\Psi) = \psi'\psi}$.
By Corollary~\ref{ref:localzeromeansglobalzero:cor} we have ${\Psi'\big(\Psi(G)\big) =
0}$. Since ${F = \Psi(G)}$ we get that ${\Psi'(F) = 0}$, thus $Z_{\psi(g)}$ is apolar to
$F$.
\end{proof}

\subsection{The local cactus rank of a general cubic surface}\label{cubicsurface}

In this subsection we restrict to characteristic $0$. First we present an example of a quartic polynomial whose cubic tail has
more partials than the polynomial itself.  Similar examples play a role in our computation of the local cactus rank of a general cubic surface, the main issue in this section.

\begin{example}
    Let ${f = x_1^{\:2}x_2 + x_2^{\:2}}$ and ${g = x_1^{\:4} + f}$. Then
    $\df(f)$ has a basis
    \[
        f,\ x_1^{\:2} + x_2,\ x_1x_2,\ x_1,\ x_2,\ 1,
    \]
    thus ${\dim_\KK\df(f) = 6}$. On the other hand, $\df(g)$ is spanned by
    \[
        g,\ x_1^{\:3} + x_1x_2,\ x_1^{\:2} + x_2,\ x_1,\ 1,
    \]
    so that ${\dim_\KK \df(g) = 5}$. Notice that
   $y_2 - y_1^{\:2}\in g^{\perp},$ so that $x_2$ does not appear in the leading summand of any partial of $g$ (cf.\ Proposition \ref{coex}).
\end{example}
For our computation of the local cactus rank of a general cubic surface $V(F)$, we need to translate the generality assumptions on $F$ into properties of its partials.
First note the following algebraic--geometric correspondences for $\Psi\in
\SVdual_1$:
\begin{enumerate}
    \item $\Psi^3 (F) = 0$ if and only if $[\Psi]\in V(F) \subseteq
        \mathbb{P}(\SVdual_1)$,
    \item $\Psi^2 (F) = 0$ if and only if $[\Psi]$ is a singular point of
        the (hyper)surface $V(F)$,
    \item $\Psi (F) = 0$ if and only if $[\Psi]$ is a cone point of $V(F)$.
\end{enumerate}
\begin{lemma}\label{general}
    Let $F\in S$ be a general cubic form in four variables. Then
    \begin{enumerate}
    \item The set of $[\Psi]\in
    \PP(\SVdual_1)$ such that the quadric $\Psi (F)$ has rank less than $4$ is an irreducible surface of degree $4$.
Furthermore, the quadric $\Psi (F)$ has rank less than $4$ if and only if there exist a $\Psi'\in \SVdual_1$ such that $(\Psi\Psi') (F)=0$.
    \item There are no points $[\Psi]\in V(F)\subset
    \PP(\SVdual_1)$, such that $\Psi (F)$ is a quadric of rank less than $3$
    i.e.\ if $\Psi_3^{\:3} (F)=0$ and $(\Psi_2\Psi_3) (F)=0$, then $\Psi_2^{\:3} (F)\neq0$.
       \item The cubic surface $V(F)$ is smooth, i.e.\ $\Psi^2 (F)\neq0$ for every nonzero $\Psi\in \SVdual_1$.
   \item The cubic surface $V(F)$ has no Eckhardt points, i.e.\ no plane section is a cone.
 \end{enumerate}

 \end{lemma}
\begin{proof} These facts are classical. For a good recent reference see
    \cite[9.4]{Dol}.\end{proof}

\begin{proposition} \label{cubic surface} Let $F\in \SV_3$ be a general smooth cubic form in four variables.
Then the local cactus rank of  $F$ is $7$.
For every linear form $l\in \SV_1$, the apolar scheme of the dehomogenization $\pi_{[l]}(F)$ has length $8$, while
if $\{l=0\}$ defines a singular curve section of $V(F)\subset\PP(\SVdual_1)$ whose tangent cone at the singular point is a square,
 then there is a length $7$ scheme supported at $[l]\in \PP(\SV_1)$ that is apolar to $F$.
\end{proposition}
\begin{proof}
Let $F\in \SV = \KK_{\divp}[x_0, x_1, x_2, x_3]$ be a general cubic form in the
    sense of Lemma~\ref{general}.

    We claim that for all non-zero $l\in \SV_1$,  the Hilbert function
    \begin{equation}\label{dimensionpartialsF_l}
   H_{F_l}= (1,3,3,1).
    \end{equation}
    Indeed, suppose it is not so, then ${H_{F_l}\le (1,3,2,1)}$ for some non-zero $l\in
    \SV_1$ and there exists a non-zero linear form $\Psi\in
        l^{\perp}$, such that $\psi  (F_l)$ has degree at most one,
        where $\psi = \pi^*(\Psi)$.
    Let ${F_l = f_3 + f_2 + f_1 + f_0}$ be the decomposition into
    homogeneous components, then ${F = f_3 + lf_2 + l^2f_1 + l^3f_0}$. Since
    ${\deg\psi (F_l) \le1}$ we have $\Psi (f_3) = 0$. But then $\Psi (F)$ is
    divisible by $l$,  so it is a quadric of rank at most $2$.  On the other hand ${\Psi(l)=0}$, so
    ${\Psi^2\big(\Psi(F)\big)=\Psi^3 (F)=0}$, so ${[\Psi]\in V(F)\subset \PP(\SVdual_1)}$. This contradicts the
    generality assumption Lemma~\ref{general}.2 of $F$.

    The cubic surface $V(F)$ has a one dimensional family of plane cuspidal cubic sections, and finitely many reducible plane sections that are unions of a smooth conic and a tangent line.  In either case, the tangent cone at the singular point is a square. We pick one such plane section.
    After a linear change of coordinates, we may assume that this plane section is $V(F,x_0)$, that it is singular at $V(x_0,x_1,x_2)$ and that  $V(x_0,x_1^{\:2})$ is the tangent cone, so that
the dehomogenization $\pi_{[x_0]}(F)$ has the form
    \[
   \pi_{[x_0]}(F)=f+x_1^{\:2}x_3+x_3^{\:2}+x_3l
    \]
    where $f$ and $l$ are polynomials in $\KK_{\divp}[x_1,x_2]$ of degree three and one respectively.
    The plane section $V(F,x_0)=V(f_3+x_1^{\:2}x_3,x_0)$ where $f_3$ is the
    cubic summand of $f$. By Lemma~\ref{general}.4, the linear partials of
    $f_3 + x_1^2x_3$ fill $\langle x_1, x_2, x_3\rangle$. For later use we
    note that $y_1^2$ is the only monomial quadric such that $y_1^2(f_3 + x_1^2x_3)$
    contains $x_3$. Hence we have ${(y_2, y_3)\cdot (y_1, y_2, y_3) (f)} =
    \langle 1, x_1,
    x_2\rangle$.
    Since $l - y_1^2(f)\in
    \KK_{\divp}[x_1, x_2]$ we can find ${\sigma_0\in (y_2, y_3)(y_1, y_2,
    y_3)}$ such that
    $\sigma_0 (f) = l - y_1^2(f) \mod \langle x_1\rangle$. Then there also exists
$\sigma = \sigma_0 + ay_1^3\in (y_1, y_2, y_3)^2$ such that $\sigma(x_1^4 + f)
= l - y_1^2(f)$. Clearly $\sigma_0 x_1^4 = 0$.

  Let $G=x_1^{\:4}+x_0F$. Then ${y_0  (G)=F}$, hence $G^{\perp} \subseteq F^{\perp}$.
   By Lemma~\ref{ref:containment:lem} we
    have $I(\ZGl) \subseteq G^{\perp}$, so we conclude that $I(\ZGl) \subseteq F^{\perp}$, i.e.\ that the local Gorenstein scheme $\ZGl$ is apolar to $F$.
    We claim that $\mathrm{length}(\ZGl) \leq 7$, and hence that the local cactus rank of $F$ is at most~$7$.

   We prove the claim, by showing that $x_1^{\:2}x_3+x_3^{\:2}+x_3l$ is an exotic summand for $G_{x_0}:=\pi_{[x_0]}(G)=x_1^{\:4}+f+x_1^{\:2}x_3+x_3^{\:2}+x_3l$.
  For this we consider the partials
  \[
  y_3(G_{x_0})=x_1^{\:2}+x_3+l, \quad y_1^{\:2}(G_{x_0})=x_1^{\:2}+x_3+y_1^{\:2}(f).
  \]
  We have $(y_3 - y_1^2)(G_{x_0}) = l - y_1^2(f) = \sigma(G_{x_0})$,
  so $y_3-y_1^{\:2}-\sigma$ annihilates $G_{x_0}$. If we take
  $\psi=y_1^{\:2}+\sigma$, then $\psi^2 G_{x_0} = y_1^4(G_{x_0}) = 1$ and we may write
  \[
  G_{x_0}=x_1^{\:4}+f+x_3\psi(x_1^{\:4}+f)+x_3^{\:2}\psi^2(x_1^{\:4}+f)
  \]
  which shows, by Proposition~\ref{ref:dualautomorphism:prop}, that  
  $x_1^{\:2}x_3+x_3^{\:2}+x_3l$ is an exotic summand for $G_{x_0}$.  Thus $G_{x_0}$ has the Hilbert function of the binary polynomial $x_1^{\:4}+f$.
 The maximal values of this function is clearly  $(1,2,2,1,1)$, so
  this proves the claim that $\mathrm{length}(\ZGl) \leq 7$.

    Finally, suppose that there exists a local Gorenstein scheme of length at most $6$
    apolar to $F$.  It must be defined by some polynomial $g$ whose cubic tail coincides
    with $F_l$ for some $l$.
    Thus $g$ has degree at least four, and its Hilbert function is $H_g = (1, 2,1,1,1)$, $H_g= (1, 1,1,1,1)$ or $H_g = (1,
    1,1,1,1,1)$.   In the first two cases,  $g$ has  degree $4$ and the leading
    summand of $g$ is a pure power $l^{\:4}$.  Each order\h{one} partial of the cubic summand
    $F_{l,3}$ of $F_l$ is therefore proportional to $l^{\:2}$.  In particular
    $\pi^*(\Psi) (F_{l,3})=0$ for some $\Psi\in l^{\bot}\subset \SVdual_1$, so $\Psi (F)$ is divisible by $l$  contradicting, as above,
    the generality assumption on $F$.

    In the case $H_g = (1, 1, 1, 1, 1, 1)$ we see that $F_l$ is the
        degree-at-most-three part of $g$ and the standard form of $g$ is
        $x^5$. By Example~\ref{ex:11111}, we have a $4\cdot 3$-dimensional choice of
    $F_l$. Together with the choice of $l$, we obtain a $16$-dimensional
variety of possible $F$, thus such $F$ is not general.
\def\oldconclusion{
    In the last case $g$ has degree $5$.
    Again its leading summand is a pure power, say $x_1^{\:5}$, and the
    leading summand of any partial of $g$ is a pure power of $x_1$.
    Furthermore  since $f$ is a cubic whose cubic summand is not a cone, it
    must involve two variables $x_2$ and $x_3$ that are hidden variables  for
    $g$.  The cubic surface is furthermore smooth, so both $x_2^{\:2}$ and $x_3^{\:2}$
    must divide some monomial in $g$.    But, by
    \red{(??)}, then
    both $y_2 (g)$ and $y_3 (g)$ must be partials of order two of $g$
    with independent leading summands. 
\pmmtodo{Also here I think some result of section
    \ref{standardformspresentation} should be used, but the results that are
there now do not state this directly.}
    Since the leading summand of any partial
    of order two is proportional to $x_1^{\:3}$, this  is a contradiction.}
\end{proof}

\section{On the dimension of the cactus varieties  of cubic forms}\label{cactusdim}

In this section we consider polynomials with Hilbert function ${(1,m-1,m-1,1)}$ and\linebreak  ${(1,m-1,m-1,1,1)}$ and derive lower bounds on the dimension of the Cactus variety of cubic forms $\Cactus_{2m}(\Ver{3,n})$ and $\Cactus_{2m+1}(\Ver{3,n})$, respectively.

The cactus variety $\Cactus_{r}(\Ver{3,n})$ of the third Veronese embedding
$\Ver{3,n} \subset\PP(\SV_3)$ is, according to Proposition \ref{tail}, the closure of the family of cubic forms $[F]$
admitting a decomposition  $F=F_1+\cdots+F_s$ and distinct linear forms
$l_1,\ldots,l_s\in \SV_1$ and forms  $G_1,\ldots,G_s\in S$, such that
$Z_{G_1,l_1}\cup\cdots\cup Z_{G_s,l_s}$ has length at most $r$, and the
dehomogenization $f_i=\pi_{[l_i]}(F)$ of $F$ at $l_i$ is the cubic tail of the dehomogenization
of $G_i$ at $l_i$ (see Definition \ref{ref:naturalapolar:def}).

To get a lower bound on the dimension of the cactus variety,  we consider the extreme opposite to the higher secants, namely linear spaces that intersect $\Ver{3,n}$ in a local scheme. In particular we consider the closure 
\[
W_{2m}(\Ver{3,n})\subset \Cactus_{r}(\Ver{3,n}) \subset\PP(\SV_3)
\]
 of the family of cubic forms $[F]$, for which there exist a linear form $l\in \SV_1$ and a form  $G\in S$, such that $f=\pi_{[l]}(F)$ is the cubic tail of  $g=\pi_{[l]}(G)$ and the polynomial $g$ 
 has Hilbert function
$(1,m-1,m-1,1)$. We define $W_{2m+1}(\Ver{3,n})$ analogously, using the
Hilbert function $(1,m-1,m-1,1,1)$.  In the first case $g$ 
is itself a cubic polynomial, i.e. its own cubic tail, while in the second case, $g$ 
is a quartic polynomial.

\def\tail#1#2{\operatorname{Tails}_{#1}\left( #2 \right)}%
To find the dimension of $W_{r}(\Ver{3,n})$ 
 when $r=2m$, we note that it is the union over ${l\in \SV_1}$ of varieties isomorphic to the projectivisation of $V_r(n)$, the family of cubic polynomials
$f\in\KK[x_1,\dots,x_n]$ with Hilbert function $(1,m-1,m-1,1)$.
When $r=2m+1$, the variety  $W_{r}(\Ver{3,n})$ is union over $l$ of varieties isomorphic to the projectivisation of $V_r(n)$ of $\tail{r}{3,n}$, the family of cubic polynomials
$f\in\KK[x_1,\ldots,x_n]$, that are tails of polynomials $g\in\KK[x_1,\ldots
,x_n]$ with Hilbert function 
$(1,m-1,m-1,1,1)$.

\begin{example}\label{even} 
If ${f\in\KK_{\divp}[x_1,\ldots,x_n]}$ has Hilbert function ${H_f= (1, m-1, m-1, 1)}$, its only possible symmetric decomposition is
\[
H_f= \Delta = \Delta_0 = (1, m-1, m-1, 1),
\]
and therefore ${\deg f=3}$,  ${\Delta_0(1)=m-1}$ and
\[
(n_0,n_1)=(m-1,m-1).
\]
If $f_\alpha$ is a general cubic polynomial in $\KK_{\divp}[x_1,\ldots,x_{m-1}]$, and $f_\beta$ is any quadratic polynomial in  $\langle x_m,\ldots,x_n\rangle \langle x_1,\ldots,x_{m-1},1\rangle$, then $f_\alpha$ is in standard form and $f_\beta$ is an exotic summand for $f_\alpha+f_\beta$ (cf.\ Definition \ref{exotic}).  
Furthermore  $H_{f_\alpha+f_\beta}=\Delta$ and so $f_\alpha+f_\beta\in
V_{2m}(n)$. 
The subspace 
\[
\langle x_1,\ldots,x_{m-1}\rangle=\linear{f}{0}
\]
 is determined by $f$,
 and the variety of subspaces $\langle
x_1,\ldots,x_{m-1}\rangle\subset\langle x_1,\ldots,x_{n}\rangle$ has dimension
$(n-m+1)(m-1)$, so we get
\[
\dim V_{2m}(n)= \binom{m + 2}{3} +(2m-1)(n-m+1).
\]
Notice furthermore that the affine variety $V_{2m}(n)$ is a cone, so that its projectivisation has dimension one less.
\end{example}
\begin{example}\label{odd} If ${g\in\KK_{\divp}[x_1,\ldots,x_n]}$ has Hilbert function ${H_f= (1, m-1, m-1, 1, 1)}$, its only possible symmetric decomposition is 
\[
\Delta =(\Delta_0,\Delta_1)  = \big( (1, 1, 1, 1, 1), (0, m-2, m-2, 0) \big),
\]
and therefore ${\deg g=4}$,  ${\Delta_0(1)=1}$, ${\Delta_1(1)=m-2}$ and
\[
(n_0,n_1,n_2)=(1,m-1,m-1).
\]
If $f_\alpha$ is a general cubic polynomial in $\KK_{\divp}[x_1,\ldots,x_{m-1}]$, and $f_\beta$ is a cubic polynomial of the form ${l_0x_1^{\:2}+l_0^{\:2}+f_\infty}$, where $l_0\in \langle x_m,\ldots,x_n\rangle$ and  $f_\infty\in \langle x_m,\ldots,x_n\rangle \langle x_1,\ldots,x_{m-1},1\rangle$, then $x_1^{\:4}+f_\alpha$  is in standard form and $f_\beta$ is an exotic summand for 
\[
g=x_1^{\:4}+f_\alpha+f_\beta=x_1^{\:4}+f_\alpha+ l_0x_1^{\:2}+l_0^{\:2}+f_\infty.
\] Furthermore 
\[
H_{g}=\Delta_0+\Delta_1=(1, m-1, m-1, 1, 1),
\] 
and so $f_\alpha+f_\beta\in \tail{2m+1}{3,n}$.  

The flag of subspaces
\[
\langle x_1\rangle=\linear{g}{0}\subset \langle
x_1,\ldots,x_{m-1}\rangle=\linear{g}{1}\subset\langle x_1,\ldots,x_{n}\rangle
\]
is determined by $g$, and  the variety of such flags  has dimension
$m-2+(n-m+1)(m-1)$,  so we get 
 \[
 \dim \tail{2m+1}{3,n}= \binom{m + 2}{3} +(2m-1)(n-m+1)+n-1.
 \]
 Notice that, since the summand $l_0x_1^2+l_0^2$ is quadratic in the form $l_0$, the affine variety $\tail{2m+1}{3,n}$ is not a cone, so its projectivisation has the same dimension.
\end{example}

We use the Examples \ref{even} and \ref{odd} to give a lower bound on the dimension of the union of linear spaces that intersect $\Ver {3,n}$ in a local subscheme.
\begin{proposition}\label{localdim}
Let $3\leq m\leq n$.  The union $\Cactus^L_{2m}(\Ver{3,n})$ of linear spans in  $
\PP^{\binom{n+3}{3}-1}$ of local subschemes in $\Ver{3,n}$ of length $2m$ has
dimension
\[
\dim \Cactus^L_{2m}(\Ver{3,n})\geq \binom{m+2}{3}+2m(n-m)+3m-2.
\]
Let $4\leq m\leq n$.  The union $\Cactus^L_{2m+1}(\Ver{3,n})$ of linear spans in  $ \PP^{\binom{n+3}{3}-1}$ of local subschemes in $\Ver{3,n}$ of length $2m+1$ has dimension
\[
\dim \Cactus^L_{2m+1}(\Ver{3,n})\geq \binom{m+2}{3}+2m(n-m)+3m+n-2.
\]
\end{proposition}
\begin{proof}  Clearly $W_{2m}(\Ver{3,n})\subset \Cactus^L_{2m}(\Ver{3,n})$ and $W_{2m+1}(\Ver{3,n})\subset \Cactus^L_{2m+1}(\Ver{3,n})$, so we get the inequality by computing the dimension of these subvarieties. 

 Let $m>2$.  $W_{2m}(\Ver{3,n})$ is the union as $l$ varies, of projective varieties whose affine cones are isomorphic to $V_{2m}(n)$, so  $W_{2m}(\Ver{3,n})$ has dimension 
\[
\dim W_{2m}(\Ver{3,n})\leq \dim V_{2m}(n)-1+n
\]
equal to the right hand side in the lemma.
Similarly, $W_{2m+1}(\Ver{3,n})$ is the union, as $l$ varies, of varieties isomorphic to $\tail{2m+1}{3,n}$, so 
\[
\dim W_{2m+1}(\Ver{3,n})\leq \dim \tail{2m+1}{3,n}+n.
\]
In both cases the right hand side is the dimension of the given parametrization of the variety $W_{r}(\Ver{3,n})$.
To get equality, we show that the parameterization is generically one to one.

When $r$ is even, we show that  for a
general  $[F]\in W_{r}(\Ver{3,n})$ there is a unique $l$ such that $Z_{F,l}$ has length $r$.  When $r$ is odd, we show that  there is a unique $l$ such that $f=\pi_l(F)$ is the tail of a quartic polynomial $g_l$ whose apolar scheme
$Z_{g_l}$ has length $r$. 

Let  $r=2m\leq 2n$, and assume that  $[F]\in W_{r}(\Ver{3,n})$ is general.  Let $l\in \SV_1$ and $f=\pi_l(F)$ be the local polynomial of $F$ at $l$ such that $f$ has Hilbert function $(1,m-1,m-1,1)$. Then $F=F_3+lF_2$ where $F_3$ depends on $m-1$ variables.  
Therefore $V(F,l)$ is a cone inside the hyperplane $V(l)$. Let $[y]$ be a
point in the $n-m$ dimensional linear vertex of $V(F,l)$.  Then  $y(F_3)=y(l)=0$.
Furthermore, since all partials of $F_l$ of degree $1$ are partials of $F_3$,
we have $y^2(F_2)=0$.
In particular, $y^2(F)=0$ and $y(F)=l\cdot l'$, so $V(F)$ is singular at $[y]$ with a tangent cone of rank $2$.
On the other hand, if $V(F)$ is singular at $[y]$ with a tangent cone of rank $2$, then 
\[
F=F_3+lF_2+l\cdot l'x,
\]
 where $y(x)=1$, $y(F_3)=y(F_2)=y(l)=y(l')=0$ and $F_3$ has Hilbert function  $(1,m-1,m-1,1)$ for some $m\leq n$.

If $r=2m+1$, and assume
that  $[F]\in W_{r}(\Ver{3,n})$ is general. Let $l\in \SV_1$ and $f=\pi_l(F)$ be the local polynomial of $F$ at $l$ such that $f$  has Hilbert function $(1,m,m,1)$, but is the cubic tail of a quartic polynomial $g$ with Hilbert function $(1,m-1,m-1,1,1)$.  When $m\leq n$, then 
\[
g=\alpha_0 l_0^4+g_3+\alpha_1 l_0^2l'+\alpha_2 (l')^2+g_2
\]
 where $g_3$ and $l_0$ depends on $m$ variables and $l'$ is a hidden variable for $g$.
Thus 
\[
f_3=g_3+\alpha_1 l_0^2l', \quad F=g_3+\alpha_1 l_0^2l'+\alpha_2 l(l')^2+lg_2
\]
and $V(F,l)$ is a singular cubic hypersurface with a double point whose tangent cone is a square.  In fact, if $m<n$, then
$V(F,l)$ is a cone with linear vertex of dimension $n-m-1$ over such a singular hypersurface.

In both cases, if $m<n$, let $l_1,\ldots,l_{n-m}$ are general linear forms.  If $r=2m$, then the linear section $V(F, l_1,\ldots,l_{n-m})$ still has a singular point whose tangent cone has rank $2$.
 If $r=2m+1$, then the linear section  $V(F_3,l, l_1,\ldots,l_{n-m})$ is still a singular cubic hypersurface inside a  $(m-1)$-dimensional linear subspace with a non reduced tangent cone at the singular point.    
 The proof of uniqueness of $l$ may therefore in both cases be reduced to the case, when $n=m$.

The following is a classical result.
\begin{lemma}\label{eckhardt} The set of singular cubic hypersurfaces in $\PP^n, n>2$ whose tangent cone at the singular point has rank at most $2$, form a subvariety of  codimension $\binom{n-1}{2}+1$ in $\PP^{\binom{n+3}{3}-1}$, and the general member in the set has exactly one singular point.
\end{lemma}
\begin{proof}
It suffices to note that the set of singular cubic hypersurfaces in $\PP^n$
form a hypersurface in $\PP^{\binom{n+3}{3}-1}$. The general point in this
hypersurface, the discriminant, corresponds to a cubic hypersurface with a quadratic singularity, i.e.\ the tangent
cone is a quadric of rank $n$.  In the space of  quadrics of rank at most $n$,
the quadrics of rank $2$ form a subvariety of codimension $\binom{n-1}{2}$.
The two codimensions add up to the codimension in the lemma.  For uniqueness it suffices to fix a quadric $q$ of rank $2$  and a point $p$ in its vertex,and notice, by Bertini's theorem, that the general cubic hypersurface through for which $q$ is the tangent cone at $p$ is smooth elsewhere.
\end{proof}
 \begin{remark}  Notice that the codimension in lemma is consistent with the dimensions of $W_{2n}(\Ver{3,n})$.
 When $n=m$ in the proposition, we get
 \[
\dim W_{2n,n}(\Ver{3,n})=\binom{n+3}{3}-1-\binom{n-1}{2}-1.
 \]
  \end{remark}

For the case $r$ is odd, we show that
 \begin{lemma} \label{tangentsquare} The set of cubic hypersurfaces in $\PP^n, n>3$ with a singular hyperplane section whose tangent cone at the singular point is a square,  form a subvariety of  codimension $\binom{n-2}{2}-1$ in $\PP^{\binom{n+3}{3}-1}$, and when $n>4$, the general member in the set has exactly one such hyperplane section.
 \end{lemma}

\begin{proof}
Assume $V(F)$ is a general cubic of dimension $n-1>2$ with a singular hyperplane section $V(F,l)$ whose tangent cone at the singular point is a square.   Let $p\in V(F,l)$ be a singular point and $V(l,(l')^2)$ the tangent cone in $V(l)$.   We may choose coordinates $x_0,\ldots,x_{n}$, so that $l=x_0,l'=x_1$ and $p=[0:\ldots:0:1]$.  Then
 \[
 F=F_3+x_1^2x_{n}+x_{0}F_2,
 \]
  where $F_3\in \KK_{\divp}[x_1,\ldots,x_{n-1}]$ and $F_3+x_1^2x_{n}$ is the cubic form defining the singular hyperplane section  $V(F,x_0)$ and $F_2\in \KK_{\divp}[x_0,\ldots,x_{n}] $.
    Thus $F$ depends on $\binom{n+1}{3}+1+\binom{n+2}{2}$ coefficients.
 Now, $p$, $l$ and $l'$ varies in a $(n+(n-1)+(n-2))$ dimensional variety, so we get that cubics with a singular hyperplane section whose tangent cone at the singular point is a square form a variety of codimension
 \[
 \binom{n+3}{3}-\left(\binom{n+1}{3}+1+\binom{n+2}{2}\right)-3(n-1)=\binom{n+1}{2}-1-3(n-1)=\binom{n-2}{2}-1
 \]
When $n>4$, this codimension is positive.  

The forms $F$, when $F_3$ and $F_2$ vary, define a linear system of cubic hypersurfaces with base locus supported at  
$p=[0:\ldots:0:1]$.  The general member is smooth, and the tangent hyperplane section at $p$ is singular only at $p$, and the tangent cone at $p$ is a square.  If this hyperplane section is not unique with this property, there is another point $q$ distinct from $p$ such that the tangent hyperplane section at $q$ also has this property.  To count dimensions, we  fix two flags $p\in L_p\subset H_p$ and $q\in L_q\subset H_q$, and consider the space of smooth cubic hypersurfaces through $p$ and $q$, whose tangent hyperplanes are $H_p$ and $H_q$ and whose tangent cones at $p$ and $q$ are squares with support  along $L_p$ and $L_q$ respectively.   Notice that $H_p$ and $H_q$ are distinct, while $L_p\cap H_q$ may equal $L_q\cap H_p$.   This gives two cases for the dimension count. These are both similar to the dimension count above and  show that the variety of  cubic hypersurfaces with two special points as above, has positive codimension in the variety of cubics with only one such point when ${n>4}$. Therefore the last statement of the lemma follows.
\end{proof}

\begin{remark}  Notice that codimension in the lemma is consistent with the dimension of $W_{2n+1}(\Ver{3,n})$. 
When ${n=m}$ in the proposition, we get
\[
\dim W_{2n+1,n}(\Ver{3,n})=\binom{n+3}{3}-1-\binom{n-2}{2}+1.
\]
\end{remark}

We conclude that the parameterization of $W_{r}(\Ver{3,n})$ is birational for any even $r$ with\linebreak ${5<r<2n+1}$, and any odd $r$  with $8<r\leq 2n+1$, and hence that  the dimension formulas of Proposition~\ref{localdim} are the dimensions of  $W_{2m}(\Ver{3,n})$ and  $W_{2m+1}(\Ver{3,n})$ respectively.
 \end{proof}

 We rewrite the formulas for the dimensions of  $W_{2m}(\Ver{3,n})$ and  $W_{2m+1}(\Ver{3,n})$ in terms of the lengths $r=2m$ (resp. $r=2m+1$):
\[
\dim W_{r}(\Ver{3,n})=\begin{cases}
(rn+r-1)+\frac{r(r-2)(r-16)}{48}-1 &\text{if } 5<r<2n+1, r\; \text{even},\\
(rn+r-1)+\frac{(r-1)(r-3)(r-17)}{48}-2&\text{if } 8<r<2n+2, r \; \text{odd}.
\end{cases}
\]

\begin{corollary}\label{TheCorollary}
When $18\leq r\leq 2n+2$, then
\[
\dim  \Cactus_{r}(\Ver{3,n})\geq
\begin{cases}
(rn+r-1)+\frac{r(r-2)(r-16)}{48}-1 &\text{if } r\geq 18\; \text{even},\\
(rn+r-1)+\frac{(r-1)(r-3)(r-17)}{48}-2&\text{if } r\geq 19\; \text{odd}.\\
\end{cases}
\]
\end{corollary}

For each possible Hilbert function for local schemes of length $r$, one may define a variety analogous to $W_r(\Ver{3,n})$.  The dimensions of these varieties are in general not known, and this remains an obstacle to finding a precise dimension for the cactus variety  $\Cactus_{r}(\Ver{3,n})$.

\medskip

Finally we leave an open question: we know that the cactus rank of a general cubic surface equals the rank, which is $5$, while the local cactus rank is $7$ (see Proposition~\ref{cubic surface}), but we do not know whether for a larger number of variables the local cactus rank and the cactus rank agree.

\begin{question}
Is the cactus rank of a cubic form in $\KK_\divp[x_0,\ldots,x_n]$ always computed locally, when ${n\geq 8}$ and the cactus rank is at least $18$? 
\end{question}

\section*{Acknowledgements}
JJ and KR thank Jaros{\l}aw Buczy\'{n}ski for fruitful discussions and the Homing Plus
programme of Foundation for Polish Science, co-financed from European Union, Regional Development Fund for partial support of their mutual visits. PMM thanks Anthony
Iarrobino and Jerzy Weyman for the invitation and hospitality at Northeastern
University, and is grateful to Anthony Iarrobino for having introduced him to
this subject and for very fruitful discussions. He also thanks James Adler for help with language.
AB was partially supported by Project Galaad of INRIA Sophia Antipolis M\'editerran\'ee, France,
Marie Curie Intra-European Fellowships for Carrer Development (FP7-PEOPLE-2009-IEF):
``DECONSTRUCT'', GNSAGA of INDAM, Mathematical Department Giuseppe Peano of Turin, Italy, and Politecnico of Turin, Italy.
JJ is a doctoral fellow at the Warsaw Center of Mathematics and Computer
Science financed by the Polish program KNOW and by Polish National Science Center, project
2014/13/N/ST1/02640 and a member of
``Computational complexity, generalised Waring type problems and tensor
decompositions'' project within ``Canaletto'',  the executive program for
scientific and technological cooperation between Italy and Poland, 2013-2015.
PMM was partially supported by Funda\c{c}\~{a}o para a Ci\^{e}ncia
e Tecnologia, projects ``Geometria Alg\'{e}brica em Portugal'',
PTDC/MAT/099275/2008, ``Comunidade Portuguesa de Geometria Alg\'{e}brica'', PTDC/MAT-GEO/0675/2012, and sabbatical leave grant SFRH/BSAB/1392/2013, by CIMA -- Centro de Investiga\c{c}\~{a}o  em Matem\'{a}tica e Aplica\c{c}\~{o}es, Universidade de \'{E}vora, projects
PEst-OE/MAT/UI0117/2011 and PEst-OE/MAT/UI0117/2014, and by Funda\c{c}\~{a}o de Amparo \`{a} Pesquisa do Estado de S\~{a}o Paulo, grant 2014/12558--9.
KR was supported by the RCN project no 239015 ``Special Geometries''.


\begin{thebibliography}{}
	
\bibitem[Alexander, Hirschowitz 1995]{AH}
Alexander, James, Hirschowitz, Andr\'e:
Polynomial interpolation in several variables.
\emph{J. of Alg.\ Geom.}, \textbf{4} (1995), no.\ 2, pp.\ 201--222.

\bibitem[Bernardi, Brachat, Mourrain 2014]{BBM}
Bernardi, Alessandra, Brachat, J\'er$\mathrm{\hat{o}}$me, Mourrain, Bernard: 
A comparison of different notions of ranks of symmetric tensors,
 Linear Algebra and its Applications, \textbf{460}, 2014, pp.\ 205--230.

\bibitem[Bernardi, Ranestad 2012]{BR}
Bernardi, Alessandra, Ranestad, Kristian:
On the cactus rank of cubics forms,
\emph{J. Symbolic Comput.} \textbf{50} (2012), pp.\ 291--297.

\bibitem [Buczy\'{n}ska, Buczy\'{n}ski 2014]{BB}
Buczy\'{n}ska, Weronika, Buczy\'{n}ski, Jaroslaw: Secant varieties to high degree Veronese reembeddings, catalecticant matrices and smoothable Gorenstein schemes, J. Algebraic Geom.\ \textbf{23} (2014), no.\ 1, pp.\ 63--90. 

\bibitem[Casnati, Notari 2011]{CN11}
Casnati, Gianfranco, Notari, Roberto:
On the irreducibility and the singularities of the Gorenstain locus of the punctual Hilbert scheme of degree 10,
\emph{J. Pure Appl.\ Algebra} \textbf{215} (2011), no.\ 6, pp.\ 1243--1254.

\bibitem[Casnati, Notari 2013]{CN13}
Casnati, Gianfranco, Notari, Roberto:
A structure theorem for $2$-stretched Gorenstein algebras,
arXiv:1312.2191, to apear in \emph{J.\ Commut.\ Algebra}.

\bibitem[Casnati, Jelisiejew, Notari 2015]{CJN}
Casnati, Gianfranco, Jelisiejew, Joachim, Notari, Roberto:
Irreducibility of the Gorenstein loci of Hilbert schemes via ray families,
\emph{Algebra Number Theory} \textbf{9} (2015), no.\ 7, pp.\ 1525--1570.

\bibitem [Dolgachev 2012]{Dol}
Dolgachev, Igor V.:  Classical algebraic geometry: a modern view. Cambridge University Press, (2012).
	
\bibitem [Emsalem 1978]{E} Emsalem, Jacques:
G\'{e}om\'{e}trie des points \'{e}pais,
\emph{Bull.\ Soc.\ Math.\ France} \textbf{106}, (1978), no.\ 4, pp.\ 399--416.

\bibitem [Iarrobino 1994]{I} Iarrobino, Anthony:
Associated graded algebra of a Gorenstein Artin Algebra,
\emph{Mem.\ Amer.\ Math.\ Soc.} \textbf{107}, (1994), no.\ 514, Amer.\ Math.\ Soc.\ Providence.

\bibitem [Iarrobino, Kanev 1999]{IK}
Iarrobino, Anthony, Kanev, Vassil:
Power Sums, Gorentein Algebras and Determinantal Loci,
\emph{Lecture Notes in Mathematics} \textbf{1721}, Springer-Verlag, Berlin Heidelberg New York (1999).

\bibitem [Jelisiejew 2015]{J} Jelisiejew, Joachim:
Classifying local Artinian Gorenstein algebras,
arXiv:1511.08007.

\bibitem [Macaulay 1927]{M}
Macaulay, Francis Sowerby:
Some properties of enumeration in the theory of modular systems,
\emph{Proc.\ London Math.\ Soc.} \textbf{S2-26} (1927), no.\ 1, pp.\ 531--555.

\end{thebibliography}
\end{document}